\documentclass[11pt]{amsart}
\usepackage[utf8]{inputenc}

\usepackage{amssymb,amsthm,amsmath}
\usepackage{dutchcal}
\usepackage{enumerate}
\usepackage{graphicx,xcolor}
\usepackage[hidelinks]{hyperref}
\usepackage{verbatim}

\newcommand{\dd}{\mathrm{d}}
\newcommand{\E}{\mathbb{E}}
\newcommand{\1}{\textbf{1}}
\newcommand{\R}{\mathbb{R}}

\newcommand{\Z}{\mathbb{Z}}
\newcommand{\C}{\mathbb{C}}
\newcommand{\mb}[1]{\mathbb{{#1}}}
\newcommand{\mc}[1]{\mathcal{{#1}}}
\newcommand{\p}[1]{\mathbb{P}\left( #1 \right)}

\newcommand{\po}[2]{\frac{\textrm{d} #1}{\textrm{d} #2}}
\newcommand{\e}{\varepsilon}
\newcommand{\kk}{k}
\newcommand{\scal}[2]{\left\langle #1, #2 \right\rangle}

\DeclareMathOperator{\vol}{vol}

\DeclareMathOperator{\proj}{Proj}
\DeclareMathOperator{\var}{Var}
\DeclareMathOperator{\conv}{conv}
\DeclareMathOperator{\sgn}{sgn}

\usepackage[paper=a4paper, left=2cm, right=2cm, top=2cm, bottom=2cm]{geometry}

\makeatletter
\def\thm@space@setup{%
  \thm@preskip=12pt plus 0pt minus 0pt
  \thm@postskip=0pt plus 0pt minus 0pt
}
\makeatother

\newtheorem{theorem}{Theorem}
\newtheorem{lemma}[theorem]{Lemma}

\theoremstyle{remark}
\newtheorem{remark}[theorem]{Remark}

\newtheorem{conjecture}{Conjecture}

\theoremstyle{definition}

\title{Extremal sections and projections of certain convex bodies: a survey}

\author{Piotr Nayar}

\address{(PN) University of Warsaw, 02-097 Warsaw, Poland.}
\email{nayar@mimuw.edu.pl}

\author{Tomasz Tkocz}

\address{(TT) Carnegie Mellon University; Pittsburgh, PA 15213, USA.}
\email{ttkocz@andrew.cmu.edu}

\thanks{PN's research supported in part d by the National Science Centre, Poland, grant 2018/31/D/ST1/0135. TT's research supported in part by NSF grant DMS-2246484.}

\date{\today}

\begin{document}

\begin{abstract} 
We survey results concerning sharp estimates on volumes of sections and projections of certain convex bodies, mainly $\ell_p$ balls, by and onto lower dimensional subspaces. This subject emerged from geometry of numbers several decades ago and since then has seen development of a variety of probabilistic and analytic methods, showcased in this survey.
\end{abstract}

\maketitle

\bigskip

\begin{footnotesize}

\end{footnotesize}

\bigskip

\section{Introduction}

\subsection{Prologue}
\emph{How small can the volume of a slice of the unit cube be?} This question, asked by Good in the 70s in the context of its applications in geometry of numbers has turned out to be rather influential, prompting development of several important methods, as well as spurring further problems and research directions of independent interest in convex geometry, with strong ties to probability. Those most notably include the \emph{dual} question of extremal-volume projections, which in the simplest nontrivial case of hyperplane-projections, naturally translates into probabilistic Khinchin-type inequalities. Intriguingly, questions on extremal-volume sections can be similarly  translated into the same probabilistic language. 

The purpose of this survey is thus two-folded: in addition to striving to give a systematic account of the known results, our second goal is to illustrate intertwined Fourier analytic, geometric and probabilistic methods underpinning the old and recent approaches.

\subsection{The motivating example}
We begin with recalling Good's question (following \cite{Ball-surv, Vaa}). Suppose we are given $n$ linear forms $L_i(x) = \sum_{j=1}^k a_{ij}x_j$, $i = 1, \dots, n$ in $k$ variables. When does the system $|L_i(x)| \leq 1$, $i \leq n$, admit a nontrivial integral solution? The cornerstone result in geometry of numbers, Minkowski's (first) theorem provides a link to volume: if $K$ is a symmetric convex body in $\R^d$ of volume at least $2^d$, then it contains a nontrivial lattice point (see, e.g. Chapter 2 in \cite{Mat}). Let $A = [a_{ij}]_{i \leq n, j \leq k}$ be the $n \times k$ matrix whose $i$th row determines $L_i$. Thus immediately, if $k \geq n$ and $\det(A) \leq 1$ when $k=n$, then the answer to Good's question is affirmative because the set 
\[
K = \{x \in \R^k, \ |L_i(x)| \leq 1, i \leq n\} = \{x \in \R^k, \ Ax \in [-1,1]^n\}
\]
is the preimage of the cube $[-1,1]^n$ under the linear map $A\colon \R^k \to \R^n$ (unbounded if $A$ is singular and of volume exactly $2^k\det(A)^{-1}$ otherwise when $k=n$). The case $k < n$ is more interesting. Suppose $A$ is of full rank $k$. Then the image of $K$ under $A$ is the section of the cube $[-1,1]^n$ by the $k$-dimensional linear subspace $A(\R^k)$. How small can its volume be? Good's conjecture confirmed later by Vaaler in \cite{Vaa} says that it is at least $2^k$ (the volume of the $k$-dimensional subcube $[-1,1]^k \times \{0\}^{n-k}$). Thus if $\det(A^\top A) \leq 1$, we obtain
\[
\vol(K) \geq \sqrt{\det(A^\top A)}\vol(K) = \vol(A(K)) \geq 2^k
\]
also asserting in view of Minkowski's theorem that the initial system of inequalities admits a nontrivial integral solution, provided the convenient sufficient condition
$
\det(A^\top A) \leq 1.
$
From a geometric point of view, it now seems natural and interesting to ask further questions about the maximal-volume sections for the cube, as well as other sets.

\subsection{Preliminaries and overview}
We endow $\R^n$ with the standard inner product $\scal{x}{y} = \sum_{j=1}^n x_jy_j$ between two vectors $x = (x_1, \dots, x_n)$ and $y = (y_1, \dots, y_n)$ in $\R^n$ and denote by $|x| = \sqrt{\scal{x}{x}}$, the induced standard Euclidean norm. Its closed centred unit ball is denoted $B_2^n$ and for the unit sphere, we write $S^{n-1} = \partial B_2^n$. Moreover, we write $e_1, \dots, e_n$ for the standard basis vectors, $e_1 = (1,0,\dots,0)$, $e_2 = (0,1,0,\dots,0)$ etc. As usual, for a set $A$ in $\R^n$, $A^\perp = \{x \in \R^n, \scal{x}{a} = 0 \ \forall a \in A\}$ is its orthogonal complement, with the convention that for a vector $u$ in $\R^n$, $u^\perp = \{u\}^\perp$ is the hyperplane perpendicular to $u$. Dilates are denoted by $\lambda A = \{\lambda a, \ a \in A\}$ for a scalar $\lambda$. In particular, if $-A = A$, the set $A$ is called (origin) symmetric. The Minkowski or algebraic sum of two sets is $A+B = \{a+b, \ a \in A, b \in B\}$. The orthogonal projection onto an affine or linear subspace $H$ in $\R^n$ is denoted by $\proj_H$. Volume, i.e. $k$-dimensional Lebesgue measure in $\R^n$ is denoted by $\vol_k(\cdot)$, identified with $k$-dimensional Haussdorff measure (normalised so that cubes with side-length $1$ have volume $1$). Recall that a body in $\R^n$ is a compact set with nonempty interior. For a symmetric convex body $K$ in $\R^n$, its Minkowski functional is $\|x\|_K = \sup\{t \geq 0, \ x \in tK\}$, $x \in \R^n$, the norm whose unit ball is $K$. A function $f\colon \R^n \to \R_+$ is called log-concave, if it is of the form $e^{-V}$ for a convex function $V\colon \R^n \to (-\infty,+\infty]$. We refer for instance to the monographs \cite{AGM, BGVV}.

To put it fairly generally, given a body $B$ in $\R^n$ and $1 \leq k \leq n$, the two questions of our main interest will be

\vspace*{4pt}
\hspace*{0.05\textwidth}
\fbox{
\begin{minipage}{0.4\textwidth}
\centering
(I) What are the minimal and maximal volume \emph{sections} $\vol_k(B \cap H)$ among all $k$-dimensional subspaces $H$ in $\R^n$?
\end{minipage}}
%
\hspace*{0.05\textwidth}
\fbox{\begin{minipage}{0.4\textwidth}
\centering
(II) What are the minimal and maximal volume \emph{projections} $\vol_k(\proj_H(B))$ among all $k$-dimensional subspaces $H$ in $\R^n$?
\end{minipage}}
\hspace*{0.05\textwidth}
\vspace*{4pt}


We note the obvious that in contrast to (I), Question (II) does not change if we translate the body $B$. 

It is worth recalling two classical convexity-type results allowing to compare such volumes in the codimension $1$ case, $k = n-1$ (despite not yielding direct answers to these questions).

\begin{theorem}[Busemann \cite{Bus}]\label{thm:Bus}
Let $K$ be a symmetric convex body in $\R^n$. Then the function
\[
x \mapsto \frac{|x|}{\vol_{n-1}(K \cap x^\perp)}, \qquad x \neq 0,\]
extended by $0$ at $x=0$ defines a norm on $\R^n$.
\end{theorem}

The surface area measure $\sigma_K$ of a convex body $K$ in $\R^n$ is a Borel measure on the unit sphere $S^{n-1}$ defined as follows: for $E \subset S^{n-1}$, $\sigma_K(E)$ equals the volume of the part of the boundary $\partial K$ where normal vectors belong to $E$ (in other words, $\sigma_K$ is the pushforward of the $(n-1)$-dimensional Haussdorff measure on $\partial K$ via the Guass map $\nu_K\colon \partial K \to S^{n-1}$).

\begin{theorem}[Cauchy-Minkowski]\label{thm:Cauchy-Mink}
Let $K$ be a convex body in $\R^n$. Then for every unit vector $\theta \in S^{n-1}$, we have
\[
\vol_{n-1}\big(\proj_{\theta^\perp}(K)\big) = \frac{1}{2}\int_{S^{n-1}} |\scal{\theta}{\xi}| \dd\sigma_K(\xi).
\]
In particular, the function
$
x \mapsto |x|\vol_{n-1}\big(\proj_{x^\perp}(K)\big), x \neq 0,
$
extended by $0$ at $x=0$, defines a norm on $\R^n$.
\end{theorem}

Let us explain this formula in the case of polytopes.  Suppose we are give a convex polytope $P$ in $\R^n$ and we want to project it onto a hyperplane $\theta^\perp$, where $\theta$ is a unit vector. Let $\mc{F_P}$ be the set of faces of $P$. If $F \in \mc{F}_P$ then $\vol_{n-1}(\proj_{\theta^\perp}(F))= \vol_{n-1}(F)\cdot |\scal{\theta}{n(F)}|$, where $n(F)$ is the unit outer-normal vector to $F$. Note that in $\proj_{\theta^\perp}(P)$ every point is \emph{covered} two times, so one gets the following expression for the volume of projection
\[
	\vol_{n-1}(\proj_{\theta^\perp} P) = \frac12 \sum_{F \in \mc{F}_P} \vol_{n-1}(F)\cdot |\scal{\theta}{n(F)}|.
\] 
The Cauchy-Minkowski formula is a straightforward generalization of this formula to general convex bodies. For further background and proofs, we refer for instance to \cite{Gar-tom} (Theorem 8.1.10 and (A.49)).

For $p > 0$ and a vector $x = (x_1, \dots, x_n)$ in $\R^n$, we define the $\ell_p$-norm of $x$ (quasinorm, when $0 < p < 1$) and its (closed) unit ball by
\[
\|x\|_p = \left(\sum_{j=1}^n |x_j|^p\right)^{1/p}, \qquad \|x\|_\infty = \max_{j \leq n} |x_j|, \qquad \qquad B_p^n = \left\{x \in \R^n, \ \|x\|_p \leq 1\right\}.
\]
The cube $B_\infty^n = [-1,1]^n$ often warrants the more convenient volume $1$ normalisation
\[Q_n = \frac{1}{2}B_\infty^n = \left[-\frac{1}{2},\frac{1}{2}\right]^n.
\]

The known results about extremal-volume hyperplane sections and projections of $\ell_p$ balls are summarised in Tables \ref{tab:sec} and \ref{tab:proj} (that is the known answers to Questions (I) and (II) when $B = B_p^n$ and $k = n-1$). We shall discuss them and many more in detail in the next sections.

\begin{table}[!hb]
\begin{center}
\caption{Extremal-volume hyperplane sections of $\ell_p$-balls: $\min/\max_{a \in S^{n-1}} \vol_{n-1}(B_p^n \cap a^\perp)$.}
\label{tab:sec}
\begin{tabular}{c|c|c|c}
 & $0 < p < 2$ & $2 < p < \infty$ & $p=\infty$ \\\hline
$\min$ & $a = \left(\frac{1}{\sqrt{n}},\dots, \frac{1}{\sqrt{n}}\right)$ \ \ \cite{Kol} & $a = (1,0,\dots,0)$ \ \ \cite{MP} & $a = (1,0,\dots,0)$ \ \ \cite{Hadw,Hen} \\\hline
$\max$ & $a = (1,0,\dots,0)$ \ \ \cite{MP} & \textbf{?} & $a = \left(\frac{1}{\sqrt{2}},\frac{1}{\sqrt{2}},0,\dots,0\right)$ \ \ \cite{Ball} 
\end{tabular}
\end{center}
\end{table}

\vspace*{2em}

\begin{table}[!hb]
\begin{center}
\caption{Extremal-volume hyperplane projections of $\ell_p$-balls: $\min/\max_{a \in S^{n-1}} \vol_{n-1}\big(\proj_{a^\perp} (B_p^n)\big)$.}
\label{tab:proj}
\begin{tabular}{c|c|c|c}
 & $p=1$ & $1 < p < 2$ & $2 < p \leq \infty$ \\\hline
$\min$ & $a = \left(\frac{1}{\sqrt{2}},\frac{1}{\sqrt{2}},0,\dots,0\right)$ \ \ \cite{Ball-wave, Sza} & \textbf{?} & $a = (1,0,\dots,0)$ \ \ \cite{BN} \\\hline
$\max$ & $a = (1,0,\dots,0)$ \ \ [folklore] & $a = (1,0,,\dots,0)$ \ \ \cite{BN} & $a = \left(\frac{1}{\sqrt{n}},\dots, \frac{1}{\sqrt{n}}\right)$ \ \ \cite{BN} 
\end{tabular}
\end{center}
\end{table}

\subsection{Existing literature and our aim}
There is of course a vast literature on the subject. Ball's survey \cite{Ball-surv} presents stochastic comparison methods and applications of the celebrated Brascamp-Lieb inequalities to derive sharp bounds on sections. Koldobsky, Ryabogin and Zvavich's survey \cite{KRZ-surv}, as well as Koldobsky's monograh \cite{Kol-mon} bring a common Fourier-analytic treatment to bounds on both sections and projections. The aim of this paper is to update on these and gather in one place what we know up-to-date, as well as highlight what we would like to know around Questions (I) and (II), presenting 11 conjectures. We also showcase a unifying probabilistic point of view (via Khinchin-type inequalities) which goes hand in hand with the Fourier-analytic methods, allowing to obtain additional insights and sharper results.

\section{Sections}\label{sec:sec}

The goal here is to give a comprehensive account of known results concerning Question (I), with some indication of methods to which we come back in Section~\ref{sec:methods}. We begin with some general remarks. A convex body $K$ in $\R^n$ is called \emph{isotropic} if it has volume $1$, its barycentre is at the origin and its covariance matrix is proportional to the identity matrix, that is
\[
\vol_n(K) = 1, \qquad \int_K x \dd x = 0, \qquad \int_K x_ix_j \dd x = L_K^2\delta_{ij}.
\]
The positive propotionality constant $L_K$ is called the isotropic constant of $K$. Every convex body admits an affine image which is isotropic (diagonalising the covariance matrix). It turns out that for symmetric isotropic convex bodies, volumes of all sections of a fixed dimension are comparable.

\begin{theorem}[Hensley \cite{Hen-gen}]\label{thm:Hensley-gen}
Fix $1 \leq \kk \leq n$. There are positive constants $c_\kk, c_\kk'$ which depend only on $\kk$ such that for every symmetric convex body $K$ in $\R^n$ which is isotropic and every $\kk$-codimensional subspace $H$ in $\R^n$, we have
\[
\frac{c_\kk}{L_K} \leq \vol_{n-\kk}(K \cap H)^{1/\kk} \leq \frac{c_\kk'}{L_K}.
\]
\end{theorem}

To illustrate the key insight of Hensley's argument, let us consider the hyperplane case: we take $H = a^\perp$ for a unit vector $a$ and let
\begin{equation}\label{eq:f}
f(t) = \vol_{n-1}(K \cap (H + ta)), \qquad t \in \R.
\end{equation}
By the Brunn-Minkowski inequality, this defines a log-concave function. By the assumptions on $K$, it is even and integrates to $1$. We claim that $f(t)$ is the probability density function of the random variable $\scal{a}{X}$, where $X$ is uniform on $K$. Indeed,
\[
	\p{\sum_{i=1}^n a_i X_i \leq s} = \p{\scal{a}{X} \leq s} = \vol_{n-1}\left(\left\{x \in K: \scal{a}{x} \leq s\right\} \right) = \int_{-\infty}^s f(t) \dd t,
\]
by Fubini's theorem.  Crucially,
\begin{equation}\label{eq:sec-density}
\vol_{n-1}(K \cap a^\perp) = f(0).
\end{equation}
Since by isotropicity $\E|\scal{a}{X}|^2=L_K^2$, we have $\int_{\R} t^2 f(t) \dd t = L_K^2$. It then remains to extremise the value of $f(0)$ among all such densities. Using a ``moving mass to where it is beneficial'' type of argument (see the proof of Theorem \ref{thm:Had-Hen} below), a sharp lower bound is obtained by considering a uniform density (in higher dimensions, i.e. $\kk > 1$, isotropicity naturally dictates a uniform density on a Euclidean ball), whilst for a sharp upper bound, using convexity, the comparison is made against a symmetric exponential density (in higher dimensions, the argument is more complicated and not sharp anymore -- see Lemma 3 in \cite{Hen-gen}).

Bourgain in \cite{Bou} used the property of hyperplane sections having comparable volume to obtain bounds on maximal functions. He asked whether the isotropic constants $L_K$ over all $K$ in all dimensions are uniformly bounded by a universal constant  -- equivalently, whether every (symmetric) convex body of volume $1$ admits a hyperplane section of volume at least a universal constant. This has become one of the central questions in asymptotic convex geometry, the hyperplane, or slicing conjecture, see e.g. \cite{BGVV} for a comprehensive monograph, \cite{KM} for a recent survey and \cite{JLV, KL, Kla} for the best results up to date.

By Theorem \ref{thm:Hensley-gen}, for two arbitrary subspaces $H_1, H_2$ of codimension $\kk$, we have
\[
\left(\frac{\vol_{\kk}(K \cap H_1)}{\vol_{\kk}(K \cap H_2)}\right)^{1/\kk} \leq C_\kk
\]
with $C_\kk = \frac{c_\kk'}{c_\kk}$. Hensley's proof gives an upper bound on $C_\kk$ of the order $\kk!$, which was improved to $\sqrt{\kk}$ by Ball in \cite{Ball-lc}, who conjectured an optimal bound to be in fact of the constant order, which remains open and turns out to be equivalent to the slicing conjecture. Implicit in his work and elucidated by V. Milman and Pajor in their seminal work \cite{MilPaj} is the following reason for that equivalence: for a symmetric isotropic convex body $K$ in $\R^n$ and a $\kk$-codimensional subspace $H$ in $\R^n$, there is a symmetric $\kk$-dimensional convex body $C$ in $H^\perp$ such that
\[
c_1\frac{L_C}{L_K} \leq \vol_{n-\kk}(K \cap H)^{1/\kk} \leq c_2\frac{L_C}{L_K},
\]
where $c_1, c_2 > 0$ are universal constants. The body $C$ emerges from a generalisation of Busemann's Theorem \ref{thm:Bus} to higher codimensions (see \cite{Ball-lc} and \cite{MilPaj}). Since $L_K \geq L_{B_2^n}$, if the slicing conjecture is true, then $L_C \leq c_3$ for a universal constant $c_3 > 0$ thus it in particular implies the existence of a universal constant $M > 0$ such that for all $\kk$-codimensional subspaces $H$, we have
\begin{equation}\label{eq:slicing-consequence}
\vol_{n-\kk}(K \cap H) \leq M^\kk.
\end{equation}

\subsection{Cube}
Recall $Q_n = [-\frac12, \frac12]^n$. As highlighted in the introduction, in the context of extremal-volume sections, it has always been the cube sparking most interest and attention. The first sharp result concerned minimum-volume hyperplane sections and was obtained independently by Hadwiger in \cite{Hadw} and Hensley in \cite{Hen}.

\begin{theorem}[Hadwiger \cite{Hadw}, Hensley \cite{Hen}]\label{thm:Had-Hen}
For every unit vector $a$ in $\R^n$, we have
\[
\vol_{n-1}(Q_n \cap a^\perp) \geq 1.
\]
Equality holds if and only if $a = \pm e_j$ for some $1\leq j \leq n$.
\end{theorem}
\begin{proof}
Let $K=Q_n$ and let us consider the function $f$ from \eqref{eq:f}. Since $Q_n$ is isotropic, the value of  $\int t^2 f(t) \dd t = \int_{Q_n} \scal{x}{a}^2 \dd x = \frac{1}{12}$ does not depend on $a$ (easily found by taking $a=e_1$). Moreover, $\|f\|_\infty = f(0)$, for $f$ is even and log-concave. It is therefore enough to show that for \emph{every} probability density $f$, we have
\[
\|f\|_\infty^2 \int t^2 f(t) \dd t \geq \frac{1}{12}.
\]
This goes back to Moriguti's work \cite{Mor} (rederived by Ball in \cite{Ball} in a slightly more general case of $p$-norms). For the proof, we can assume that $f$ is even, as otherwise we consider $g(t)=\frac12(f(t)+f(-t))$ and $\|g\|_\infty\leq \|f\|_\infty$, whereas the second moments of $f$ and $g$ are the same. Then we move mass towards the origin, as this is beneficial: formally, take $f_0=\|f\|_\infty \1_{[-c,c]}$, where $c=(2 \|f\|_\infty)^{-1}$. Clearly $\|f_0\|_\infty = \|f\|_\infty$. We have
\[
	\int t^2(f(t)-f_0(t)) \dd t  = \int (t^2-c^2)(f(t)-f_0(t)) \dd t \geq 0,
\]  
as the integrant is nonnegative.
\end{proof}

In words, the canonical coordinate subspaces uniquely minimise the volume of hyperplane sections of the cube. Soon after, this was extended to sections of arbitrary dimension by Vaaler in \cite{Vaa}, confirming Good's conjecture.

\begin{theorem}[Vaaler \cite{Vaa}]\label{thm:vaaler}
Fix $1 \leq k \leq n$. For every $k$-dimensional subspace $H$ in $\R^n$, we have
\[
\vol_{k}(Q_n \cap H) \geq 1.
\]
Equality holds if and only if $H$ is spanned by some $k$ standard basis vectors.
\end{theorem}

Thus every section of the cube has \emph{large} volume. It is also ``fat in all directions'', in terms of quadratic forms, see Ball and Prodromou's work \cite{BPr}.
Vaaler used log-concavity and the notion of peakedness (introduced by Kanter in \cite{Kant}) to make such comparison, generalising Hensley's argument. Recently, Akopyan, Hubard and Karasev in \cite{AHK} gave a different proof based on topological methods. 

Thus Vaaler's theorem gives a complete answer to Question (I) for mimimal-volume sections of the cube. Turning to the maximal ones, the first general upper bound for hyperplane sections was given by Hensley in \cite{Hen}, viz.
$
\vol_{n-1}(Q_n \cap a^\perp) \leq 5
$
for every unit vector $a$ in $\R^n$, who also conjectured that the sharp bound would be with $5$ replaced by $\sqrt{2}$ attained at $a = (\frac{1}{\sqrt{2}},\frac{1}{\sqrt{2}},0,\dots,0)$. This was later confirmed by Ball in his seminal work \cite{Ball}.

\begin{theorem}[Ball \cite{Ball}]\label{thm:ball-codim1}
For every unit vector $a$ in $\R^n$, we have
\[
\vol_{n-1}(Q_n \cap a^\perp) \leq \sqrt{2}.
\]
Equality holds if and only if $a = (\pm e_i \pm e_j)/\sqrt{2}$ for some $1\leq i < j \leq n$.
\end{theorem}

\begin{proof}[Sketch of the proof]
The starting point of Ball's approach was Fourier-analytic: if we fix a unit vector $a$ and let $f$ be the probability density of $\scal{a}{X}$, where $X$ is a random vector uniform on $Q_n$ (thus having i.i.d. components $X_j$ which are uniform on $[-\frac12, \frac12]$), then by \eqref{eq:sec-density} and the standard Fourier inversion formula,
\[
\vol_{n-1}(Q_n \cap a^\perp) = f(0) =  \frac{1}{2\pi}\int_{-\infty}^\infty \hat f(t) \dd t =  \frac{1}{2\pi}\int_{-\infty}^\infty \prod_{j=1}^n \frac{\sin(\frac{1}{2}a_jt)}{\frac{1}{2}a_jt} \dd t 
=  \frac{1}{\pi}\int_{-\infty}^\infty \prod_{j=1}^n \frac{\sin(a_jt)}{a_jt} \dd t.
\]
(This formula can perhaps be traced back to P\'olya's work \cite{Pol}, and was also used by Hensley.) The next crucial idea is to apply H\"older's inequality with the weights $p_j = a_j^{-2}$ to get the bound
\begin{equation}\label{eq:Ball-Holder}
\int_{-\infty}^\infty \prod_{j=1}^n \frac{\sin(a_jt)}{a_jt} \dd t \leq \prod_{j=1}^n \left(\int_{-\infty}^\infty \left|\frac{\sin(a_jt)}{a_jt}\right|^{a_j^{-2}} \dd t\right)^{a_j^2} = \prod_{j=1}^n \Psi\left(a_j^{-2}\right)^{a_j^2}
\end{equation}
with 
$
\Psi(p) = \sqrt{p}\int_{-\infty}^\infty \left|\frac{\sin t}{t}\right|^{p}\dd t, p \geq 1
$
(a similar trick was also used in Haagerup's seminal work \cite{Haa} on sharp constants in Khinichin inequalities). The most technically challenging and rather intricate is the problem of maximisation of $\Psi$. The so-called Ball's integral inequality which he established to finish his proof asserts that
\begin{equation}\label{eq:Ball-integralineq}
\frac{1}{\pi}\int_{-\infty}^\infty \left|\frac{\sin t}{t}\right|^{p}\dd t \leq \sqrt{\frac{2}{p}}, \qquad p \geq 2
\end{equation}
with equality if and only if $p = 2$. This completes the proof in the case where all $|a_j| \leq \frac{1}{\sqrt{2}}$. The complimentary case is dispensed with by a geometric argument justifying that $\vol_{n-1}(Q_n \cap a^\perp) \leq \frac{1}{|a_j|}$ for each $j$. Indeed, projecting the section onto $e_j^\perp$ changes its volume by the factor $|\scal{a}{e_j}|$ and it is contained in the projection of the entire cube,
\[
\vol_{n-1}(Q_n \cap a^\perp) = \frac{1}{|\scal{a}{e_j}|}\vol_{n-1}(\proj_{e_j^\perp}(Q_n \cap a^\perp)) \leq \frac{1}{|\scal{a}{e_j}|}\vol_{n-1}(\proj_{e_j^\perp}(Q_n)) = \frac{1}{|a_j|}.
\]
\end{proof}

We mention in passing that this integral inequality has been quite influential, with a very powerful method developed by Nazarov and Podkorytov in \cite{NP} to give a ``simple'' proof, as well as many extensions, generalisations, discrete versions, or even stability properties (see \cite{ALM, EM, KOS, LPP, MR, MelRob}). 

Quite remarkably and unexpectedly, the $\sqrt{2}$ bound allows to produce a very simple counter-example to the famous Busemann-Petty problem posed in \cite{BP}: \emph{If for two symmetric convex bodies $K$ and $L$ in $\R^n$, we have $\vol_{n-1}(K \cap a^\perp) \leq \vol_{n-1}(L \cap a^\perp)$ for every vector $a$, does it follow that $\vol_n(K) \leq \vol_n(L)$?} Indeed, Ball observed in \cite{Ball-BP} that since the volume of the hyperplane sections of the unit volume Euclidean ball in high dimensions is roughly $\sqrt{e}$ and $\sqrt{2} < \sqrt{e}$, it suffices to take $K = Q_n$ and for $L$, the ball of a slightly smaller radius. This argument in fact works in all dimensions $n \geq 10$. Later in \cite{Gia1}, Giannopoulos used similar ideas involving cylinders to produce such elegant and simple counter-examples in dimensions $n \geq 7$. The answer to the Busemann-Petty problem is negative for $n \geq 5$ and positive for $n \leq 4$, see \cite{GKS} which has culminated efforts of significant work. We refer for instance to Koldobsky's comprehensive monography \cite{Kol-mon} for a full account.

The situation for upper bounds on volume of sections of codimensions higher that~$1$ is not fully understood. Ball has obtained two general bounds.

\begin{theorem}[Ball \cite{Ball2}]\label{thm:ball-codim>1}
Fix $1 \leq k \leq n$. For every $k$-dimensional subspace $H$ in $\R^n$, we have
\[
\vol_{n-1}(Q_n \cap H) \leq \min\left\{\sqrt{\frac{n}{k}}^{k},\sqrt{2}^{n-k}\right\}.
\]
\end{theorem}

The first bound $\sqrt{\frac{n}{k}}^{k}$ is optimal when $k$ divides $n$. Rogalski's question asked for the symmetric convex body of largest volume ratio (see \cite{Ball2}). It turns out that the bound $\sqrt{\frac{n}{k}}^{k}$ is equivalent to the fact that the cube is such a maximiser. The second bound $\sqrt{2}^{n-k}$ is better than the first one for $k \geq \frac{n}{2}$ and turns out to be sharp in this case. Both bounds rely heavily on Ball's ingenious geometric version of the Brascamp-Lieb inequality (from \cite{BL}), which provides a multidimensional analog of H\"older's inequality. The first bound uses it in a direct way (applied to indicator functions of intervals), whereas the second bound applies it to the Fourier analytic formula. As already mentioned, the exponential bound $\sqrt{2}^{n-k}$ in the codimension $n-k$ is largely motivated by the slicing problem, see \eqref{eq:slicing-consequence}, providing an explicit constant for the cube.

These bounds, although sharp for many values of $k$ and $n$, leave many other cases open. A sort of folklore conjecture (see, e.g. \cite{Iva, Poo}) states that for arbitrary $k$ and $n$, the maximal-volume section of the cube is attained at an affine cube. Specifically, given $1 \leq k \leq n$, let $n = k\ell+r$ with $r$ being the remainder from the division of $n$ by $k$. We define the following $k$ orthogonal vectors in $\R^n$:
\begin{align*}
u_{j+1} &= e_{j\ell+1} + e_{j\ell+2} + \dots + e_{(j+1)\ell}, \hspace*{15em} 0 \leq j < k-r, \\
u_{k-r+j} &= e_{(d-r)\ell + (\ell+1)j + 1} + e_{(d-r)\ell + (\ell+1)j + 2} + \dots + e_{(d-r)\ell + (\ell+1)(j+1)}, \quad 0 \leq j < r.
\end{align*}
Let $H^*$ be the $k$-dimensional subspace spanned by them. Then
\[
Q_n \cap H^* = \left\{\sum_{j=1}^k t_ku_k, \ |t_1|, \dots, |t_k| \leq \frac{1}{2}\right\}
\]
which is an affine cube of volume
\[
\vol_k(Q_n \cap H^*) = \prod_{j=1}^k |u_j| = \sqrt{\ell}^{k-r}\sqrt{\ell+1}^{r}.
\]
Note that this becomes $\sqrt{n/k}^{k}$ when $k$ divides $n$ and $\sqrt{2}^{n-k}$ when $k \geq n/2$. 

\begin{conjecture}\label{conj:min-sec-cube}
Let $1 \leq k \leq n$. For every $k$ dimensional subspace $H$ in $\R^n$, we have
\[
\vol_k(Q_n \cap H) \leq \vol_k(Q_n \cap H^*).
\]
\end{conjecture}

In addition to Ball's results of Theorem \ref{thm:ball-codim>1}, this conjecture has recently been confirmed for planar sections, i.e. when $k=2$ by Ivanov and Tsiutsiurupa in \cite{IT}, who developed local conditions for extremal subspaces.

At the end of this subsection, we mention several loosely related extensions of these fundamental results.

\subsubsection*{Other measures.}
Let $\gamma_n$ denote the standard Gaussian measure on $\R^n$, that is the Borel probability measure on $\R^n$ with density $(2\pi)^{-n/2}e^{-|x|^2/2}$, whereas for a  subspace $H$, let $\gamma_H$ be its counterpart on $H$, that is the Borel probability measure supported on $H$ with density $(2\pi)^{-\dim H/2}e^{-|x|^2/2}$ (with respect to Lebesgue measure on $H$). Due to the lack of homogeneity, now of course cube's side lengths may play a role. For the lower bounds, Barthe, Gu\'edon, Mendelson and Naor in \cite{BGMN} established an analogue of Vaaler's theorem.

\begin{theorem}[Barthe-Gu\'edon-Mendelson-Naor \cite{BGMN}]\label{thm:BGMN-cube-low}
Fix $1 \leq k \leq n$. For every $k$-dimensional subspace $H$ in $\R^n$, the function
\[
t \mapsto \frac{\gamma_H(tQ_n \cap H)}{\gamma_k(tQ_k)}
\]
is nonincreasing on $[0,+\infty)$. In particular (letting $t \to \infty$), for every $t > 0$, we have
\[
\gamma_H(tQ_n \cap H) \geq \gamma_k(tQ_k).
\]
\end{theorem}

Their argument follows Vaaler's approach, crucially using the product structure of Gaussian measure. Using Ball's geometric form of the Brascamp-Lieb inequality, they also obtain an upper bound, similar to his bound for volume.

\begin{theorem}[Barthe-Gu\'edon-Mendelson-Naor \cite{BGMN}]\label{thm:BGMN-cube-up}
Fix $1 \leq k \leq n$. For every $k$-dimensional subspace $H$ in $\R^n$ and every $t > 0$, we have
\[
\gamma_H(tQ_n \cap H) \leq \gamma_k\left(t\sqrt{\frac{n}{k}}Q_k\right).
\]
\end{theorem}

Again, this is sharp whenever $k$ divides $n$. 
The maximal-Gaussian-volume hyperplane sections of cubes are not known for all values of $t$. Zvavitch has showed in \cite{Zva} that the hyperplane $(\frac{1}{\sqrt{2}},\frac{1}{\sqrt{2}},0,\dots,0)^\perp$ cannot be extremal for all dilates because the bound from Theorem \ref{thm:BGMN-cube-up} in the case $k=n-1$ is tight as $t \to \infty$. K\"onig and Koldobsky in \cite{KoKol-prod} have found conditions on product measures assuring that the hyperplane $(\frac{1}{\sqrt{2}},\frac{1}{\sqrt{2}},0,\dots,0)^\perp$ gives the maximal volume among all hyperplanes $a^\perp$ with $\max_j |a_j| \leq \frac{1}{\sqrt{2}}$. When specialised to the standard Gaussian measure, they have additionally obtained that the hyperplane $(\frac{1}{\sqrt{2}},\frac{1}{\sqrt{2}},0,\dots,0)^\perp$ yields maximal volume (among \emph{all} hyperplanes) if and only if  $t < t_0 = 1.253..$. Sharp upper bounds for $t > t_0$ are not known.

\subsubsection*{Cylinders}
Dirksen in \cite{Dirk-cyl} has studied the extremal central sections of the generalised cylinders $Z_r = Q_n \times (rB_2^m)$, $r > 0$, $m, n \geq 1$. He has found sharp upper bounds in the $3$ dimensional case of an ordinary cylinder, i.e. $m=2, n=1$, as well as upper bounds in the general case, sharp for large radii, developing Fourier analytic formulae and delicate integral inequalities involving Bessel functions.

\subsubsection*{Perimeter}
Answering a question of Pe\l czy\'nski about hyperplane sections of maximal-\emph{perimeter} (i.e. sections with the boundary of the cube), K\"onig and Koldobsky in \cite{KK-perim} have shown that the extremal direction is the same as for the volume.

\begin{theorem}[K\"onig-Koldobsky \cite{KK-perim}]\label{thm:KK-perim}
Let $n \geq 3$. For every unit vector $a$ in $\R^n$, we have
\[
\vol_{n-2}(\partial Q_n \cap a^\perp) \leq 2((n-2)\sqrt{2}+1).
\]
\end{theorem}

This bound is attained if $a = \left(\pm e_i \pm e_j\right)/\sqrt{2}$ for some $1\leq i < j \leq n$. This theorem also leads to counter-examples to a permiter version of the Busemann-Petty probem in dimensions $n \geq 14$. For the proof, they derive a Fourier analytic formula for the perimeter; its analysis involves new ingredients, most notably local conditions for constrained extrema, as well as subtle technical estimates around Ball's integral inequality.

\subsubsection*{Diagonal sections}
Here we consider the volume of the section by the hyperplane perpendicular to the main diagonal
\[
\alpha_n = \vol_{n-1}\big(Q_n \cap (\underbrace{1,\dots,1}_{n})^\perp\big), \qquad n \geq 1.
\]
Perhaps a more natural interpretation of the sequence $\alpha_1, \dots, \alpha_n$ is as the volumes of the sections of $Q_n$ by hyperplanes perpendicular to the diagonals of subcubes of growing dimension, for $1 \leq k \leq n$, we have
\[
\vol_{n-1}\big(Q_n \cap (\underbrace{1,\dots,1}_{k},\underbrace{0,\dots,0}_{n-k})^\perp\big) = \alpha_k.
\]
Theorems \ref{thm:Had-Hen} and \ref{thm:ball-codim1} in particular assert that $\alpha_1 \leq \alpha_i \leq \alpha_2$. Interestingly, the volumes of the diagonal sections form a (strictly) increasing sequence.

\begin{theorem}[Bartha-Fodor-Gonz\'alez \cite{BFG}]\label{thm:diagonals}
We have,
$
\alpha_1 < \alpha_3 < \alpha_4 < \alpha_5 < \dots < \alpha_2.
$
\end{theorem}

Their approach starts with P\'olya's formula $\alpha_n = \frac{\sqrt{n}}{\pi}\int_{-\infty}^\infty \left(\frac{\sin t}{t}\right)^n \dd t $ and is based on an intricate asymptotic analysis by means of the Laplace method. They first argue that the sequence $(\alpha_n)$ increases for all $n \geq n_0$ for some $n_0$. Then, using numerical estimates, they bound $n_0$ and deal with all $n \leq n_0$ by computer assisted calculations. Their arguments also show that the sequence $(\alpha_n)$ is eventually concave.

It is tempting to believe that critical hyperplane sections must be diagonal, that is if $a \mapsto \vol_{n-1}(Q_n \cap a^\perp)$ has an extremum at a unit vector $a^*$ then $a^*$ is proportional to a diagonal $(1, \dots, 1, 0, \dots, 0)$. In \cite{Amb}, Ambrus, and independently Ivanov and Tsiutsiurupa in \cite{IT} have recently found an elegant local condition (with vastly different methods). Moreover, Ambrus has confirmed this believe for $n \leq 3$ and disproved it for $n = 4$. 

\subsubsection*{Discrete version}
Melbourne and Roberto in \cite{MelRob} have derived a sharp discrete analogue of Ball's upper bound for hyperplane sections.

\begin{theorem}[Melbourne-Roberto \cite{MelRob}]\label{thm:MelRob}
Let $n, \ell_1, \dots, \ell_n \geq 1$ and $t, k_1, \dots, k_n$ be integers. Then,
\[
\left|\left\{z \in \Z^n \cap \prod_{j=1}^n [k_j, k_j+\ell_j-1], \ \sum_{j=1}^n z_j = t \right\}\right| < \sqrt{2}\frac{\prod_{j=1}^n \ell_j}{\sqrt{\sum_{j=1}^n(\ell_j^2-1)}}.
\]
\end{theorem}
The constant $\sqrt{2}$ is best possible as can be seen by discretizing Ball's extremiser (by taking $\ell_1 = \ell_2 = m$, $\ell_3 = \dots = \ell_n = 1$ and letting $m \to \infty$). Mimicing Ball's approach, the following integral inequality lies at the heart of the argument
\[
\int_{-1/2}^{1/2} \left|\frac{\sin(n\pi t)}{n\sin(\pi t)}\right|^p \dd t < \sqrt{\frac{2}{p(n^2-1)}}, \qquad p \geq 2, n = 2, 3, \dots.
\]
This is in fact stronger than Ball's inequality \eqref{eq:Ball-integralineq} and recovers it by letting $n \to \infty$. Melbourne and Roberto have developed a new view-point on establishing such delicate bounds for oscillatory integrands, borrowing and combining ideas from majorisation and optimal transport.

\subsubsection*{Chessboard-cutting}
It is folklore that a line can meet the interiors of no more than $2N-1$ squares of the usual $N \times N$ chessboard and this bound is tight (consider the diagonal pushed down a bit). We refer to B\'ar\'any and Frenkel's work \cite{BF1} for a short argument as well as precise estimates for a $3$-dimensional analogue. To tackle the problem in higher dimensions, in \cite{BF}, they have introduced the following quantity involving volumes of hyperplane sections of the cube,
\[
V_n = \max_{v \in \R^n} \frac{\|u\|_1}{|v|}\vol_{n-1}(Q_n \cap v^\perp).
\] 
They have shown that if the cube $[0,N]^n$ is divided into $N^n$ unit cubes in the usual way, then the maximal number of the unit cubes that a hyperplane can intersect equals
\[
(1+o(1))V_nN^{n-1}
\] 
for a fixed $n \geq 1$ as $N \to \infty$. Confirming a conjecture from \cite{BF}, Aliev  has recently found the constant $V_n$ in \cite{Ali}.

\begin{theorem}[Aliev \cite{Ali}]\label{thm:aliev}
Let $n \geq 1$. We have, $V_n = \sqrt{n}\vol_{n-1}(Q_n \cap (1,\dots,1)^\perp)$.
\end{theorem} 

In words, it is the diagonal section that maximises $V_n$, thus $\sqrt{n} \leq V_n \leq \sqrt{2}\sqrt{n}$ and $V_n \sim \sqrt{\frac{6}{\pi}}\sqrt{n}$ for large $n$. Aliev's argument is purely geometric with the main observation being that the hyperplane parallel to $(1,\dots,1)^\perp$ supports the intersection body of the cube.

\subsubsection*{Stability} 
With additional insights gained from a certain probabilistic point of view (see Section \ref{sec:methods}), Chasapis and the authors have recently obtained in \cite{CNT} a dimension free stability result for both lower and upper bounds for hyperplane sections.

\begin{theorem}[Chasapis-Nayar-Tkocz \cite{CNT}]\label{thm:cube-stab}
There are universal constants $c_1, c_2 > 0$ such that for every unit vector $a$ in $\R^n$ with $a_1 \geq \dots \geq a_n \geq 0$, we have
\[
1 + c_1|a - e_1|^2 \leq \vol_{n-1}(Q_n \cap a^\perp) \leq \sqrt{2} - c_2\left|a-\frac{e_1+e_2}{\sqrt{2}}\right|.
\]
\end{theorem}

The exponents $2$ and $1$ on the left and the right hand side respectively are best possible, as can be explicitly verified for $n=2$. In an independent work \cite{MR}, Melbourne and Roberto have obtained a similar result.

\subsection{Balls of $p$-norms}
We begin with a monotonicity result in the parameter $p$ discovered by Mayer and Pajor in \cite{MP}.

\begin{theorem}[Meyer-Pajor \cite{MP}]\label{thm:MP}
Fix $1 \leq k \leq n$. For every $k$-dimensional subspace $H$ in $\R^n$, the function
\[
p \mapsto \frac{\vol_k(B_p^n \cap H)}{\vol_k(B_p^k)}
\]
is nondecreasing on $[0,+\infty)$. In particular, comparing against the Euclidean ball yields that
\begin{align*}
\vol_{k}(B_p^n \cap H) &\leq \vol_{k}(B_p^k), \qquad 0 < p < 2, \\
\vol_{k}(B_p^n \cap H) &\geq \vol_{k}(B_p^k), \qquad p > 2.
\end{align*}
In each inequality, the equality holds if and only if $H$ is spanned by some $k$ standard basis vectors.
\end{theorem}

Meyer and Pajor established this theorem for $p \geq 1$, which was extended later to $p < 1$ independently by Barthe in \cite{Bar} and Caetano in \cite{Cae}. Letting $p \to \infty$ recovers Vaaler's Theorem \ref{thm:vaaler} for the cube-sections. Vaaler's argument uses Kanter's peakedness to make a comparison between the uniform and Gaussian distribution. The key point in \cite{MP} was that the same comparison holds across the whole family of probability measures with densities $\left\{e^{-c_p|x|^p}\right\}_{p > 0}$. We will present this crucial idea in a probabilistic setting in Section \ref{sec:methods}.

More is known for hyperplane sections when $0 < p < 2$. Meyer and Pajor in \cite{MP} have found that the minimal-volume hyperplane sections of the cross-polytope $B_1^n$ are attained by the diagonal directions and conjectured the same for the entire range $0 < p < 2$, confirmed later by Koldobsky in \cite{Kol} in a strong Schur-convexity-type result. 

\begin{theorem}[Koldobsky \cite{Kol}]\label{thm:Kol}
Let $0 < p < 2$. For every two unit vectors $a$ and $b$ in $\R^n$ such that $(b_1^2, \dots, b_n^2)$ majorises $(a_1^2, \dots, a_n^2)$, we have
\[
\vol_{n-1}(B_p^n \cap a^\perp) \leq \vol_{n-1}(B_p^n \cap b^\perp).
\]
\end{theorem}

For background on majorisation and Schur-convexity, we refer for instance to \cite{Bh}. In particular, since
\[
\left(\frac{1}{n}, \dots, \frac{1}{n}\right) \prec (a_1^2,\dots,a_n^2) \prec (1,0,\dots,0),
\]
for an arbitrary unit vector $a$ in $\R^n$, the minimal and maximal volume sections follow. What makes the range $0<p<2$ so much more tractable compared to $p>2$ is the fact that the Fourier transform of $e^{-|x|^p}$ is a nonnegative function of the form $t \mapsto \int_0^\infty e^{-ut^2} \dd \mu(u)$, a Gaussian mixture. In fact the same also holds for $e^{-|x|^p}$, which allowed \cite{ENT1} to bypass the Fourier-analytic arguments entirely. We return to this in Section \ref{sec:methods}.

The maximal-volume hyperplane sections of $B_p^n$-balls for $2 < p < \infty$ are unknown. Oleszkiewicz has established in \cite{Ole} that Ball's upper bound for the cube, Theorem \ref{thm:ball-codim1}, does not extend to all $p > 2$, as it fails for all $p < 26.265..$ and large enough dimensions (by comparing cube's extremising hyperplane $(1,1,\dots,0)^\perp$ to the diagonal one $(1,1\dots, 1)^\perp$ in the limit $n \to \infty$). We conjecture that in each dimension there is a unique phase-transition point.

\begin{conjecture}\label{conj:min-sec-Bpn}
For every $n \geq 3$, there is a unique $p_0(n)$ such that
\[
\max_{a \in S^{n-1}} \vol_{n-1}(B_p^n \cap a^\perp) = \begin{cases} \vol_{n-1}(B_p^n \cap (1,\dots, 1)^\perp), & 2 < p \leq p_0(n), \\ \vol_{n-1}(B_p^n \cap (1,1,0,\dots, 0)^\perp), & p \geq p_0(n). \end{cases}
\]
\end{conjecture}


For lower dimensional sections, there is a general bound of Barthe which extends a corresponding result for the cube from Theorem \ref{thm:ball-codim>1}. The argument also crucially relies on the Brascamp-Lieb inequalities.

\begin{theorem}[Barthe \cite{Bar}]\label{thm:Barthe}
Let $p \geq 2$. Fix $1 \leq k \leq n$. For every $k$-dimensional subspace $H$ in $\R^n$, we have
\[
\vol_k(B_p^n \cap H) \leq \left(\frac{n}{k}\right)^{k(1/2-1/p)}\vol_k(B_p^k).
\]
\end{theorem}

As for the cube, this is sharp when $k$ divides $n$ with the same extremising subspace. 

Using a direct argument involving triangulation and convexity of certain functions, Nazarov has shown that planar sections of the cross-polytope of minimal area are attained at regular polygons.

\begin{theorem}[Nazarov, in \cite{CNT}]\label{thm:Nazarov}
Let $n \geq 3$. For every $2$-dimensional subspace $H$ in $\R^n$, we have
\[
\vol_2(B_1^n \cap H) \geq \frac{n^2\sin^3\left(\frac{\pi}{2n}\right)}{\cos\left(\frac{\pi}{2n}\right)},
\]
which is optimal, attained when $B_1^n \cap H$ is a regular $2n$-gon,
\end{theorem}

All known results from Table \ref{tab:sec} on extremal-volume hyperplane sections for $\ell_p$-balls admit robust versions (recall also Theorem \ref{thm:cube-stab}).

\begin{theorem}[Chasapis-Nayar-Tkocz \cite{CNT}]\label{thm:Bpn-stab}
For every $p > 0$, there is a positive constant $c_p$ such that for every $n\geq 1$ and every unit vector $a = (a_1, \dots, a_n)$ in $\R^n$ with $a_1 \geq a_2 \geq \dots \geq a_n \geq 0$, we have
\begin{align*}
\label{eq:st-Bp-p<2-max}
\frac{\vol_{n-1}(B_p^n \cap a^\perp)}{\vol_{n-1}(B_p^{n}\cap e_1^\perp)} &\geq 1  +  c_p|a - e_1|^2, \hspace*{6em} 2 < p \leq \infty,\\
\frac{\vol_{n-1}(B_p^n \cap a^\perp)}{\vol_{n-1}(B_p^n \cap (\frac{e_1+\dots+e_n}{\sqrt{n}})^\perp)} &\geq 1 + c_p\sum_{j=1}^n (a_j^2-1/n)^2, \hspace*{3em} 0 < p < 2,\\
\frac{\vol_{n-1}(B_p^n \cap a^\perp)}{\vol_{n-1}(B_p^{n}\cap e_1^\perp)} &\leq \left(a_1^p+(1-a_1^2)^{p/2}\right)^{-1/p}, \qquad 0 < p < 2.
\end{align*}
\end{theorem}

We finish this subsection with Vaaler's conjecture on general rather precise lower bounds which have been verified to a large extent for $\ell_p$-balls. 

\begin{conjecture}[Vaaler \cite{Vaa}]\label{conj:Vaa-min-sec}
Let $K$ be a symmetric isotropic convex body in $\R^n$. Then for every nonzero subspace $H$ in $\R^n$ of dimension $1 \leq k \leq n$, we have
\[
\vol_k(K \cap H) \geq 1.
\]
\end{conjecture}

Noteworthy, if true, it implies the slicing conjecture (made independently of it), see Hensley's Theorem \ref{thm:Hensley-gen}. Vaaler's theorem confirms this inequality for the cube (which is tight). Meyer and Pajor's sharp lower bound gives this inequality for $K = B_p^n$ with $2 < p < \infty$ and all subspaces (see \cite{MP}), as well as $1 < p < 2$ and all hyperplanes (see Schmuckenschl\"ager's note \cite{Sch}), however these are not tight anymore.

\subsection{Simplex}
Here we discuss results concerning sections of regular simplices. It will be most convenient to consider a regular $n$-dimensional simplex of side length $\sqrt{2}$ embedded in $\R^{n+1}$,
\[
\Delta_n = \left\{x \in \R^{n+1}, \ x_1, \dots, x_{n+1} \geq 0, \ \sum_{j=1}^{n+1} x_j = 1\right\}.
\]
\emph{Central} sections will refer to those by (affine) subspaces passing through the barycentre $\left(\frac{1}{n+1},\dots, \frac{1}{n+1}\right)$ of $\Delta_n$. In particular, if $a$ is a unit vector in $\R^{n+1}$ with $\sum_{j=1}^{n+1} a_j = 0$ (so parallel to the hyperplane containing $\Delta_n$), then $\Delta_n \cap a^\perp$ is a central hyperplane section of $\Delta_n$. Such sections of maximal volume have been determined by Webb in \cite{Webb}.

\begin{theorem}[Webb \cite{Webb}]\label{thm:Webb}
For every unit vector $a$ in $\R^{n+1}$ with $\sum_{j=1}^{n+1} a_j = 0$, we have
\[
\vol_{n-1}(\Delta_n \cap a^\perp) \leq \frac{1}{\sqrt{2}}\frac{\sqrt{n+1}}{(n-1)!}.
\]
This is attained if and only if $a^\perp$ passes through some $n-1$ vertices of $\Delta_n$.
\end{theorem}

Webb gave two proofs, both based on an elegant probabilistic formula,
\[
\vol_{n-1}(\Delta_n \cap a^\perp) = \frac{\sqrt{n+1}}{(n-1)!}f_a(0),
\]
where $f_a$ is the probability density of $\sum_{j=1}^{n+1} a_jX_j$ with the $X_j$ being i.i.d. standard exponential random variables, that is with density $e^{-x}$ supported on $(0,+\infty)$. Thus, his result becomes $f_a(0) \leq \frac{1}{\sqrt{2}}$ with equality if and only if $n-1$ of the $a_j$ vanish. His first proof mimicked Ball's Fourier analytic approach with the crucial bound coming from H\"older's inequality and an integral inequality. His second proof was probabilistic, exploiting log-concavity.

Webb has also found that the $1$ and $2$-dimensional central sections of $\Delta_n$ of maximal volume are attained at lines and planes passing through a vertex and an edge of $\Delta_n$ respectively (see his PhD thesis \cite{Webb-thesis}, as well as \cite{LiuTko} for a different argument in the line case).

For general upper bounds on central sections, following the approach involving Ball's geometric form of the Brascamp-Lieb inequality, Dirksen in \cite{Dirk} has obtained the following result.

\begin{theorem}[Dirksen \cite{Dirk}]\label{thm:Dirksen-simplex}
For every $k$-dimensional subspace of $\R^{n+1}$ passing through the barycentre of the simplex $\Delta_n$, we have
\[
\vol_{k-1}(\Delta_n \cap H) \leq \frac{k^{\frac{k}{2(n+1)}}}{(k-1)!}.
\]
Moreover, if $\mathrm{dist}(H,e_j) \leq \sqrt{\frac{n+1-k}{n+2-k}}$ for each $j \leq n+1$, then
\[
\vol_{k-1}(\Delta_n \cap H) \leq \frac{1}{(k-1)!}\sqrt{\frac{n+1}{n+2-k}}
\]
which is sharp, attained when $H$ contains $k-1$ vertices of $\Delta_n$.
\end{theorem}

As opposed to symmetric convex bodies for which maximum volume sections by all affine subspaces of a fixed dimension always occur when they pass through the barycentre (by the Brunn-Minkowski inequality), for the simplex such a question becomes nontrivial. Webb pointed out in \cite{Webb} that combining two results of Ball yields that for fixed $1 \leq k \leq n$, we have
\[
\vol_k(\Delta_n \cap H) \leq \vol_k(F_k),
\]
for all $(k+1)$-dimensional affine subspaces $H$ in $\R^{n+1}$, where $F_k$ is a $k$-dimensional face of $\Delta_n$, that is the $k$-dimensional slices of $\Delta_n$ of maximal volume are exactly the $k$-dimensional faces. To explain this, fix $H$, and consider the maximum volume ellipsoid, say $\mathcal{E}^*$ contained in the convex body $K = \Delta_n \cap H$. Ball has found in \cite{Ball-vr} that the $n$-simplex has maximal volume ratio among all convex bodies in $\R^n$. The volume ratio of a convex body $C$ in $\R^n$ is $\text{vr}(C) = \left(\vol_n(C)/\vol_n(\mathcal{E})\right)^{1/n}$, where $\mathcal{E}$ is the maximum volume ellipsoid in $C$. Thus,
\[
\vol_k(\Delta_n \cap H) = \text{vr}(K)^k \vol_k(\mathcal{E}^*) \leq \text{vr}(F_k)^k \vol_k(\mathcal{E}^*),
\]
Moreover, Ball has shown in \cite{Ball-simplex} that among all $k$-dimensional ellipsoids in $\Delta_n$, the Euclidean balls inscribed in $k$-faces have maximal volume, thus they are the maximal volume ellipsoids in $F_k$. Therefore, 
\[
\text{vr}(F_k)^k \vol_k(\mathcal{E}^*) \leq \vol_k(F_k).
\]
In \cite{Fra}, Fradelizi has given a different argument, deriving this fact from a more general result for cones in isotropic position.

Lower bounds are much less understood. 

\begin{conjecture}\label{conj:simplex-low-bd}
For every unit vector $a$ in $\R^{n+1}$ with $\sum_{j=1}^{n+1} a_j = 0$, we have
\[
\vol_{n-1}(\Delta_n \cap a^\perp) \geq \left(\frac{n}{n+1}\right)^{n-1/2}\frac{\sqrt{n+1}}{(n-1)!}
\]
which is attained when $a^\perp$ is parallel to a face of $\Delta_n$.
\end{conjecture}

This has been confirmed in low dimensions $n \leq 4$ by Dirksen in \cite{Dirk}. Brzezinski in \cite{Brz} noticed that a bound of the correct order but off by a multiplicative constant follows by applying Fradelizi's theorem from \cite{Fra} to Webb's result stated above.

\subsection{Complex analogues}
If we consider $\C^n$ as Hilbert space equipped with the standard (complex) inner product and volume (Lebesgue measure after the natural identification $\C^n \simeq \R^{2n}$), most of the results about extremal-volume sections (of real spaces) considered thus far beg for their natural complex counterparts. 
Vaaler's theorem as well as its generalisation of Meyer and Pajor admit such extensions, with almost the same proofs, as was pointed out in their papers.

\begin{theorem}[Vaaler \cite{Vaa}, Meyer-Pajor \cite{MP}]\label{thm:Val-MP}
Let $1 \leq k \leq n$ and let $H$ be a (complex) $k$-dimensional subspace in $\C^n$. Then
\[
\vol_{2k}(B_{p,\C}^n \cap H) \geq \vol_{2k}(B_{p,\C}^k),
\]
when $2 \leq p \leq \infty$. The reverse inequality holds when $0 < p \leq 2$.
\end{theorem}

Here, 
\[
B_{p,\C}^n = \left\{z \in \C^n, \ \left(\sum_{j=1}^n|z_j|^p\right)^{1/p} \leq 1\right\}
\]
is the unit ball of the complex $\ell_p(\C^n)$ space, in particular $B_{\infty, \C}^\infty$ is the polydisc (the Cartesian product of the unit discs in $\C$). In fact, their proofs yield a further extension from $B_{p,\C}^n$ to bodies which are $\ell_p$-sums of Euclidean spaces of arbitrary dimensions, which has been in turn significantly generalised by Eskenazis in \cite{Esk} (see Theorem \ref{thm:Esk} below). 

Ball's cube-slicing result of Theorem \ref{thm:ball-codim1} has been extended to the complex setting by Oleszkiewicz and Pe\l czy\'nski in \cite{OP}, who have proved the following sharp polydisc-slicing bound.

\begin{theorem}[Oleszkiewicz-Pe\l czy\'nski \cite{OP}]\label{thm:OP}
For every unit vector $a$ in $\C^n$, we have
\[
\vol_{2n-2}(B_{\infty, \C}^n \cap a^\perp) \leq 2\pi^{2n-2}.
\]
Equality holds if and only if $a = (\xi e_i + \eta e_j)/\sqrt{2}$ for some $1\leq i < j \leq n$ and $\xi, \eta \in \C$ with $|\xi|, |\eta| = 1$.
\end{theorem}

The proof strategy follows the same path of the Fourier analytic formula and  defactorisation by means of H\"older's inequality, however new technical challanges arise: the heart of the proof is the following analytical inequality
\begin{equation}\label{eq:OP-integralineq}
\int_0^\infty \left|\frac{2J_1(t)}{t}\right|^p t\dd t \leq \frac{4}{p}, \qquad p \geq 2,
\end{equation}
cf. \eqref{eq:Ball-integralineq}, where $J_1$ is the Bessel function of the first kind of order $1$. Its proof rests on precise pointwise bounds on $J_1$ as well as an interpolation argument. A new different proof has been very recently given in \cite{MelRob}. Moreover, the upper bounds for higher codimensions of Theorem \ref{thm:ball-codim>1} can be transfered almost verbatim to the complex case as well (as was remarked by Barthe and Koldobsky, see \cite{OP}). A stability version has been recently established in \cite{GTW}.

The exact analogue of the sharp upper bound on the perimeter from Theorem \ref{thm:KK-perim} also holds, as shown by K\"onig and Koldobsky in \cite{KK-perim}.

Sharp upper bounds even on hyperplane (complex codimension $1$) sections in the range $2 < p < \infty$ remain open. For the same reasons as in the real case, the range $0 < p < 2$ is more tractable and we have the following analogue of Koldobsky's Theorem \ref{thm:Kol}.

\begin{theorem}[Koldobsky-Zymonopoulou \cite{KZ}]\label{thm:KolZym}
Let $0 < p < 2$. For every two unit vectors $a$ and $b$ in $\C^n$ such that $(|b_1|^2, \dots, |b_n|^2)$ majorises $(|a_1|^2, \dots, |a_n|^2)$, we have
\[
\vol_{2n-2}(B_{p,\C}^n \cap a^\perp) \leq \vol_{2n-2}(B_{p,\C}^n \cap b^\perp).
\]
\end{theorem}

Finally, a complex version of Busemann's Theorem \ref{thm:Bus} has been developed by Koldobsky, Paouris and Zymonopoulou in \cite{KPZ}, whereas a full solution to the complex Busemann-Petty problem is due to Koldobsky, K\"onig and Zymonopoulou, \cite{KKZ}.

\subsection{Miscellanea} 
We finish this section with a brief account on various results related to and motivated by sharp bounds on volumes of sections.

\subsubsection*{Slabs}
For a unit vector $a$ in $\R^n$ and $t > 0$, we set
\[
H_{a,t} = \{x \in \R^n, \ |\scal{x}{a}| \leq t\}
\]
to be the (symmetric) slab of width $2t$ orthogonal to the direction $a$ (in other words, a thickening/enlargement $a^\perp + tB_2^n$ of the hyperplane $a^\perp$).
Answering a question of V. Milman, Barthe and Koldobsky in \cite{BK} have established the following extension of Hadwiger and Hensley's Theorem~\ref{thm:Had-Hen}.

\begin{theorem}[Barthe-Koldobsky \cite{BK}]\label{thm:BK}
For every unit vector $a$ in $\R^n$ and $0 \leq t \leq \frac{3}{8}$, we have
\[
\vol_n(Q_n \cap H_{a,t}) \geq \vol_n(Q_n \cap H_{e_1, t}).
\]
\end{theorem}

They have derived this from a sharp inequality for unimodal log-concave densities in one dimension, expanding on Hensley's approach.

In words, Hadwiger and Hensley's result is stable in that, independent of the dimension, \emph{coordinate slabs} contain least volume of the unit cube among all symmetric slabs of fixed width at most $3/4$.  This bound is in the spirit of the concentration of measure (see \cite{BLM, Led, LT}), providing a sharp lower bound of \emph{small} enlargements on the volume $1/2$ half-spaces $\{x \in \R^n, \scal{x}{a} \geq 0\}$ in $Q_n$. The threshold $\frac{3}{8}$ is suboptimal: in dimension $2$, a direct calculation from \cite{BK} shows that at $t = \sqrt{2}-1$ the extremising slab changes from the coordinate one to the diagonal one.  The sharp behaviour in higher dimensions is not clear. The paper \cite{BK} provides asymptotic results that the slabs orthogonal to the main diagonal are optimal for \emph{large} $t$ of the order $\sqrt{n}$ as $n\to \infty$ (developing en route very interesting conditions for convexity properties of Laplace transforms), with a precise nonasymptotic result for the range $\frac12\sqrt{n-1} \leq t \leq \frac{1}{2}\sqrt{n}$ obtained recently by Moody, Stone, Zach and Zvavitch in \cite{MSZZ}. 

A detailed analysis of the (local as well as global) extremal slabs in dimensions $2$ and $3$ has been done by K\"onig and Koldobsky in \cite{KK-slabs}, whereas in \cite{KK-Cslabs}, they have obtained a complex analogue of Theorem \ref{thm:BK}. A sharp probabilistic extension of Theorem \ref{thm:BK} has recently been obtained in \cite{HTW}.

\subsubsection*{Block subspaces}
Eskenazis in \cite{Esk} has gathered under one umbrella the results on slicing $\ell_p$-balls, both real and complex, when $0 < p < 2$, thus significantly generalising and unifying Theorems \ref{thm:Kol}, \ref{thm:Val-MP} and \ref{thm:KolZym}.

\begin{theorem}[Eskenazis \cite{Esk}]\label{thm:Esk}
Let $m, n$ be positive integers and let $0 < p < 2$. Suppose $X = (\R^m, \|\cdot\|)$ is a quasi-normed space which admits an isometric embedding into $L_p$. For every two unit vectors $a$ and $b$ in $\R^n$ such that $(b_1^2, \dots, b_n^2)$ majorises $(a_1^2, \dots, a_n^2)$, we have
\[
\vol_{mn-m}(B_p^n(X) \cap H_a) \leq \vol_{mn-m}(B_p^n \cap H_b).
\]
\end{theorem}
Here,
\[
B_p^n(X) = \left\{x=(x_1, \dots, x_n) \in \R^m \times\dots\times \R^m, \ \left(\sum_{j=1}^n \|x_j\|^p\right)^{1/p} \leq 1 \right\}
\]
is the unit ball of the $\ell_p$ sum of $X$, whereas
\[
H_a = \left\{x=(x_1, \dots, x_n) \in \R^m \times\dots\times \R^m, \ \sum_{j=1}^n a_jx_j = 0\right\}
\]
is a \emph{block} subspace of codimension $m$ in $(\R^m)^n$. In particular $X = \ell_2^m$ with $m = 1, 2$ recovers Theorems \ref{thm:Kol} (when $p < 2$), \ref{thm:Val-MP} and \ref{thm:KolZym}. The point is that there is a plethora of non-Hilbertian examples treated by this general result, most notably $X = \ell_q^m$ with $p \leq q \leq 2$.

Eskenazis' argument builds on \cite{ENT1}, with the key new ingredient being Lewis' representation guaranteeing that the norm on $X$ which embeds isometrically into $L_p$, $p > 0$, admits a form
\[
\|x\| = \left(\int_{S^{m-1}}|\scal{Ux}{\theta}|^p \dd \mu(\theta)\right)^{1/p}, \qquad x \in \R^m,
\]
for some invertible linear map $U\colon \R^m \to \R^m$ and an isotropic Borel measure $\mu$ on the unit sphere $S^{m-1}$, \cite{Lew, SchZva}. The restriction $p < 2$ is not needed here, but to bring about Gaussian mixtures (as highlighted after Theorem \ref{thm:Kol}).

For the regime $p > 2$, only the case $p=\infty$, $X = \ell_2^m$ has been considered, i.e. sections of
\[
B_\infty^n(\ell_2^m) = \underbrace{B_2^m \times \dots \times B_2^m}_n,
\]
for which Brzezinski in \cite{Brz} has obtained that for every $n, m \geq 2$ and every unit vector $a$ in $\R^n$, we have
\begin{equation}\label{eq:Brzez}
\vol_{mn-m}(B_\infty^n(\ell_2^m) \cap H_a) \leq \frac{(m+2)^{m/2}}{2^{m/2-1}m\Gamma(m/2)}.
\end{equation}
This is asymptotically sharp as $n \to \infty$ because the right hand side equals exactly $\lim_{n\to \infty} \vol_{mn-m}(B_\infty^n(\ell_2^m) \cap H_{(\frac{1}{\sqrt{n}},\dots,\frac{1}{\sqrt{n}})})$. The case $m=2$ is special in that this limit also equals $A_{m,n} = \vol_{mn-m}(B_\infty^n(\ell_2^m) \cap H_{(\frac{1}{\sqrt{2}},\frac{1}{\sqrt{2}},0,\dots,0})$, whilst for every $m > 2$, the limit is strictly larger than $A_{m,n}$. In other words, Ball's upper bound from Theorem \ref{thm:ball-codim1} does \emph{not} generalise to block-subspace sections of $B_2^m \times \dots \times B_2^m$ for any $m > 2$ (but it does when $m=2$ as we have seen in Oleszkiewicz and Pe\l czy\'nski's Theorem \ref{thm:OP}).

We finish with Eskenazis' conjecture on sharp lower bounds by block subspaces, generalising Hadwiger and Hensley's Theorem \ref{thm:Had-Hen}.

\begin{conjecture}[Eskenazis \cite{Esk}]
Let $m, n \geq 1$. Let $K$ be a symmetric convex body in $\R^m$. For every unit vector $a \in \R^n$, we have
\[
\vol_{mn-m}(K\times \dots \times K \cap H_a) \geq \vol_m(K)^{n-1}.
\]
\end{conjecture}

\subsubsection*{Noncentral sections}
In this context, perhaps the most natural question to ask is about extremal volume sections by affine subspaces at a \emph{fixed} distance $t > 0$ from the origin. This has arguably proved to be more difficult than the question of central sections, even for the cube. Sharp results for line sections have been found in \cite{MSZZ} for the cube and in \cite{LiuTko} for the cross-polytope. For hyperplane sections, we have the following conjecture of V. Milman (see \cite{KK-slabs}).

\begin{conjecture}[V.Milman]
The minimum and maximum of $\vol_{n-1}(Q_n \cap H)$ over the affine hyperplanes $H$ at a fixed distance $t > 0$ from the origin is attained when $H$ is orthogonal to a diagonal direction $(1,\dots,1,0,\dots,0)$ with a suitable number of $1$s depending on $t$.
\end{conjecture}

There are several partial results supporting it. K\"onig and Koldobsky have verified that it holds in low dimensions $n = 2, 3$, \cite{KK-slabs}. Moody, Stone, Zach and Zvavitch have established that in the range $\frac{1}{2}\sqrt{n-1} < t < \frac{1}{2}\sqrt{n}$ the main diagonal direction gives the maximal section, \cite{MSZZ}, later extended to all $t > \frac{1}{2}\sqrt{n-2}$ by Pournin in \cite{Pou1}, where one of the key ideas was to employ a noteworthy combinatorial formula for sections of the cube,
\[
\vol_{n-1}\Big([0,1]^n \cap \{x \in \R^n, \ \scal{x}{a} = b\}\Big) = \sum_v \frac{(-1)^{\sum v_j}|a|(b-\scal{v}{a})^{n-1}}{(n-1)!\prod a_j},
\]
where the sum is over the vertices $v$ of the cube $[0,1]^d$ such that $\scal{v}{a} \leq b$ (see also \cite{BS}). In a recent preprint \cite{Pou2}, Pournin has also showed that the main diagonal direction is strictly locally maximal for $t = \Omega(\frac{\sqrt{n}}{\log n})$, derived from general local conditions for all diagonal directions. K\"onig and Rudelson in \cite{KRud} have obtained dimension-free lower bounds on noncentral sections of the cube as well as the polydisc. K\"onig in \cite{Kon-noncent} has treated noncentral extremal volume as well as perimeter sections of the regular simplex, cube and cross-polytope, when distance $t$ is fairly large, also investigating local behaviour for the entire range of $t$.

\subsubsection*{Probabilistic extensions}
There is a natural link between the volume of sections and negative moments of linear forms, which goes back at least to Kalton and Koldobsky's work \cite{KalKol}. To illustrate it, first note that the value at say $x=0$ of a probability density $f$ on $\R$ which is continuous at $0$ can be obtained by taking the limit of its negative moments,
\begin{equation}\label{eq:KalKol}
f(0) = \lim_{q \to -1+} \frac{1+q}{2}\int |x|^{q}f(x) \dd x.
\end{equation}
In view of this and the basic probabilistic formula for sections \eqref{eq:sec-density}, the sharp bounds for hyperplane sections of the cube from Theorems \ref{thm:Had-Hen} and \ref{thm:ball-codim1} can be phrased as
\[
1 \leq \lim_{q \to -1+} \frac{1+q}{2}\E\left|\sum_{j=1}^n a_jU_j\right|^{q} \leq \sqrt{2}
\]
for all unit vectors $a$ in $\R^n$, where $U_1, U_2, \dots$ are i.i.d. random variables uniform on $[-\frac12, \frac12]$. Do such inequalities remain true with a fixed $q$? The answer is known for the cube and polydisc, where a sharp phase transition of the extremiser occurs for the upper bound with diagonal directions entering the picture

\begin{theorem}[Chasapis-K\"onig-Tkocz \cite{CKT}]\label{thm:CKT}
Let $-1 < q < 0$. Let $U_1, U_2, \dots$ be i.i.d. random variables uniform on $[-\frac12,\frac12]$. For every $n \geq 1$ and unit vectors $a$ in $\R^n$, we have
\[
\E|U_1|^q \leq \E\left|\sum_{j=1}^n a_jU_j\right|^{q} \leq \begin{cases} \E|(U_1+U_2)/\sqrt{2}|^q, & -1 < q \leq q_0, \\ \lim_{m\to\infty}  \E|(U_1+\dots+U_m)/\sqrt{m}|^q, & q_0 \leq q < 0. \end{cases}
\]
The constant $q_0 = -0.79..$ is given uniquely by equating the two expressions on the right hand side.
\end{theorem}

A similar behaviour has been established for the polydisc slicing by Chasapis, Singh and Tkocz in \cite{CST}, with the phase transition ``moving to the left'' where the negative moments recover volume.

\section{Projections}\label{sec:proj}
We turn our attention to Question (II) from the introduction about projections of extremal volume of basic convex bodies such as the cube, simplex and cross-polytope, as well as the family of $\ell_p$-balls. As we will see, the understanding of hyperplane projections of $\ell_p$-balls is at the same level as for sections (see Tables \ref{tab:sec} and \ref{tab:proj}), whilst in general much less is known, particularly for lower-dimensional projections. The methods also seem to shift from analytic to more of algebraic and combinatorial nature.

\subsection{Cube}
Thanks to Cauchy's formula from Theorem \ref{thm:Cauchy-Mink}, extremal volume projections on hyperplanes are easy to determine, for the surface area measure of the cube $Q_n$ is the counting measure $\sum_{j=1}^n \delta_{\pm e_j}$ of the set of the $2n$ vectors $\{\pm e_j, \ j \leq n\}$ outer normal to the facets of $Q_n$, thus for every unit vector $a$ in $\R^n$, we have
\[
\vol_{n-1}(\proj_{a^\perp}(Q_n)) = \sum_{j=1}^n |a_j|.
\] 
Therefore,
\[
1 \leq \vol_{n-1}(\proj_{a^\perp}(Q_n)) \leq \sqrt{n},
\]
by squaring and neglecting the off-diagonal terms for the lower bound and simply applying the Cauchy-Schwarz inequality for the upper bound.
The former is attained if and only if $Q_n$ is projected onto a coordinate hyperplane and the latter if and only if $Q_n$ is projected onto a hyperplane orthogonal to a main diagonal.

A zonotope is the Minkowski sum of intervals. Orthogonal projections of the unit cube $Q_n = [-\frac{1}{2},\frac{1}{2}]$ are zonotopes and, conversely, every zonotope can be obtained as such a projection (of a possibly rescaled and translated cube in a sufficiently high dimension). Shephard's decomposition of zonotopes into parallelopipeds led him in \cite{She} to the following classical formula for volume: if $v_1, \dots, v_n$ are vectors in $\R^k$, then the volume of the zonotope $Z = \sum_{j=1}^n [0, v_j]$ is expressed as
\[
\vol_k(Z) = \sum_{1 \leq j_1 < \dots < j_k \leq n} |\det[v_{j_1} \ \dots \ v_{j_k}]|,
\]
where $[v_{j_1} \ \dots \ v_{j_k}]$ is the $k \times k$ matrix with columns $v_{j_1}, \dots, v_{j_k}$. For the orthogonal projection of the cube $Q_n$ onto a $k$-dimensional subspace $H$, the vectors $v_j$ can be taken as columns of the $k \times n$ matrix whose rows form an orthonormal basis of $H$, leading to the constraint
\[
\sum_{1 \leq j_1 < \dots < j_k \leq n} |\det[v_{j_1} \ \dots \ v_{j_k}]|^2 = 1,
\]
by the Cauchy-Binnet formula. Then, exactly as in the codimension $1$ case, we obtain upper and lower bounds on the volume. This argument goes back to Chakerian and Filliman's work \cite{CF}.

\begin{theorem}[Chakerian-Filliman \cite{CF}]\label{thm:ChakFill}
Fix $1 \leq k \leq n$. For every $k$-dimensional subspace $H$ in $\R^n$, we have
\[
1 \leq \vol_{n-1}(\proj_{H}(Q_n)) \leq \min\left\{\sqrt{\binom{n}{k}}, \frac{\vol_{k-1}(B_2^{k-1})^k}{\vol_k(B_2^k)^{k-1}}\left(\frac{n}{k}\right)^{k/2}\right\}
\]
\end{theorem}

The lower bound is clearly sharp, attained at coordinate subspaces. It also instantly follows from Vaaler's Theorem \ref{thm:vaaler} upon observing that projections contain sections. The first upper bound $\sqrt{\binom{n}{k}}$ is sharp only when $k = 1, n-1$. The second upper bound is obtained differently, by invoking quermassintegrals (which are additive under Minkowski sums, so go hand in hand with zonotopes), combined with Urysohn's inequality. A simpler version of the same idea is to note that every $k$-dimensional projection has diameter at most the diameter of the cube $\sqrt{n}$, thus by the isodiametric inequality, its volume is at most $\vol_k(B_2^k)(\sqrt{n}/2)^k$. All these bounds are of the order $n^{k/2}$ for a fixed $k$ as $n \to \infty$, which is tight. The second of the upper bounds in Theorem \ref{thm:ChakFill} is asymptotically better than the first one. Ivanov in \cite{Iva-frames} has developed local conditions for maximisers of $k$-dimensional projections.

In \cite{CF}, using the isoperimetric inequality for polygons, Chakerian and Filliman additionally obtained a sharp bound for two-dimensional projections (and thus also $n-2$-dimensional ones -- see Theorem \ref{thm:McMullen} below).

\begin{theorem}[Chakerian-Filliman \cite{CF}]\label{thm:ChakFill-2dim}
For $n \geq 2$ and for every $2$-dimensional subspace $H$ in $\R^n$, we have
\[
\vol_{n-1}(\proj_{H}(Q_n)) \leq \frac{1}{\tan\left(\frac{\pi}{2n}\right)}.
\]
\end{theorem}

Soon after, Filliman in \cite{Fil1} discovered a general principle that maximising volume of the larger class of zonotopes $Z = \sum_{j=1}^n [-\frac{1}{2}v_j, \frac{1}{2}v_j]$ with the constraint $\sum_{j=1}^n |v_j|^2 = n$ on the vectors $v_j$ in $\R^k$ amounts to maximising it over all zonotopes which are $k$-dimensional projections of the cube. This allowed him to extend the previous estimate to $3$-dimensional projections.

\begin{theorem}[Filliman \cite{Fil1}]
For $n \geq 3$ and for every $3$-dimensional subspace $H$ in $\R^n$, we have
\[
\vol_{3}(\proj_{H}(Q_n)) \leq \frac{\sqrt{n/3}}{\tan\left(\frac{\pi}{2n-2}\right)}.
\]
\end{theorem}

This is trivially sharp for $n=3$, but also for $n=4$ with the extremal projection being the rhombic dodecahedron and for $n=6$ with the extremal projection being the triacontahedron.

We finish with a striking and remarkable feature of the cube: its projections onto orthogonal complementary subspaces have the same volume.

\begin{theorem}[McMullen \cite{McM}, Chakerian-Filliman \cite{CF}]\label{thm:McMullen}
Let $1 \leq k \leq n$. For every $k$-dimensional subspace $H$ in $\R^n$, we have
\[
\vol_k(\proj_H(Q_n)) = \vol_{n-k}(\proj_{H^\perp}(Q_n)).
\]
\end{theorem}

This has been found by McMullen and independently by Chakerian and Filliman, using the same approach based on Shephard's formula. In particular, sharp bounds on volumes of $k$-dimensional projections are equivalent to those on $(n-k)$-dimensional ones.

\subsection{Simplex}
Recall that $\Delta_n$ is the $n$-dimensional regular simplex with edge length $\sqrt{2}$, assuming for convenience in this section that $\Delta_n$ is embedded in $\R^n$. The projections of the regular simplex onto certain orthogonal complementary subspaces are conjectured to yield minimal and maximal volume, in huge contrast to the cube, where we have seen in Theorem \ref{thm:McMullen} that such projections always have the \emph{same} volume.

\begin{conjecture}[Filliman \cite{Fil-simplex}]
Fix $1 \leq k \leq n$. Let $H_*$ be a $k$-dimensional subspace in $\R^n$ such that $T_* = \proj_{H_*}(\Delta_n)$ is a $k$-dimensional simplex, with the vertices of $\Delta_n$ projecting only onto the vertices of $T_*$, as \emph{evenly} as possible: for each $i \leq k+1$, letting $w_i$ be the number of vertices of $\Delta_n$ projecting onto vertex $i$ of $T_*$, we have that 
\[
w_i = \begin{cases}\ell+1, & 1 \leq i \leq r, \\ \ell, & r < i \leq k+1, \end{cases}
\]
where we divide $n+1$ by $k+1$ with the remainder $r \in \{0, \dots, k\}$, $n+1 = (k+1)\ell + r$.
Then
\[
\min_{\substack{H \subset \R^n \\ \dim H = k}} \vol_k(\proj_H(\Delta_n)) = \vol_k(T_*).
\] 
Moreover, the polytope $T^* = \proj_{H_*^\perp}(\Delta_n)$ is conjectured to maximise the volume of projections onto the $n-k$ dimensional subspaces,
\[
\max_{\substack{H \subset \R^n \\ \dim H = n-k}} \vol_k(\proj_H(\Delta_n)) = \vol_{n-k}(T^*).
\]
\end{conjecture}

Filliman developed exterior algebra techniques in \cite{Fil2} and used them in \cite{Fil-simplex} to confirm this conjecture in the following cases, for the minimum: $k=1, 2, n-1$ and arbitrary $n$, as well as $n \leq 6$ and arbitrary $k$, and for the maximum: $k = 1, 2, n-1$ and arbitrary $n$, $k=n-2$ and $n \leq 8$, as well as $k=3$ and $n=6$.

\subsection{Cross-polytope} 
In view of Cauchy's formula from Theorem \ref{thm:Cauchy-Mink}, the volume of hyperplane projections of the cross-polytope $B_1^n$ admits a natural probabilistic expression. Since it has $2^n$ congruent (simplicial) facets of $(n-1)$-dimensional volume $\frac{\sqrt{n}}{(n-1)!}$ with outer-normals $\frac{1}{\sqrt{n}}(\pm 1, \dots, \pm 1)$, for every unit vector in $\R^n$, we have
\[
\vol_{n-1}(\proj_{a^\perp}(B_1^n)) = \frac{1}{2(n-1)!}\sum_{\e \in \{-1,1\}^n} |\scal{a}{\e}| = \frac{2^{n-1}}{(n-1)!}\E\left|\sum_{j=1}^n a_j\e_j\right|,
\]
where the expectation is over independent random signs $\e_j$, $\p{\e_j = \pm 1} = \frac12$. Given the constraint $|a| = 1$, the question about extremal volume projections thus becomes that of finding best constants $c, C$ in the homogeneous inequalities
\begin{equation}\label{eq:Khinchin}
c\left(\E\left|\sum a_j\e_j\right|^2\right)^{1/2} \leq \E\left|\sum a_j\e_j\right| \leq C\left(\E\left|\sum a_j\e_j\right|^2\right)^{1/2}.
\end{equation}
Such $L_p$-moments comparison inequalities go back to Khinchin's work \cite{Kh} on the law of the iterated logarithm. This motivated and should be  contrasted with an analogous probabilistic view-point on sections from Theorem \ref{thm:CKT}, where instead of the $L_{1}$-norm, we have the limit of $L_{q}$-norm as $q \downarrow -1$.
A sharp upper bound follows easily from Jensen's inequality,
\[
\E\left|\sum a_j\e_j\right| \leq \left(\E\left|\sum a_j\e_j\right|^2\right)^{1/2} = |a| = 1,
\]
attained if and only if $a = \pm e_i$ for some $i \leq n$, that is the maximum-volume projection occurs at precisely coordinate subspaces. The reverse inequality is much deeper: in a different context, a sharp lower bound had been conjectured to be attained at vectors $a = \frac{\pm e_i \pm e_j}{\sqrt{2}}$, $i \neq j$ by Littlewood in \cite{Lit} (cf. Ball's extremiser from Theorem \ref{thm:ball-codim1}), proved much later by Szarek in \cite{Sza}, with subsequently simplified and quite different proofs \cite{Haa, LO-best, Tom}. We state it here rephrased in terms of volumes of projections, together with the simple upper bound.

\begin{theorem}[Szarek \cite{Sza}]\label{thm:Szarek}
Let $n \geq 2$. For every unit vector in $\R^n$, we have
\[
\frac{1}{\sqrt{2}}\vol_{n-1}(B_1^{n-1}) \leq \vol_{n-1}(\proj_{a^\perp}(B_1^n)) \leq \vol_{n-1}(B_1^{n-1}).
\]
The lower bound is attained if and only if $a = \frac{\pm e_i \pm e_j}{\sqrt{2}}$ for some $i \neq j$, whilst the upper bound, if and only if $a = \pm e_i$ for some $i$.
\end{theorem}

A stability version has been derived by De, Diakonikolas and Servedio in \cite{DDS} (see also \cite{MR} for a local statement with explicit constants).

Much less is known in higher codimensions. In analogy to Vaaler's Theorem \ref{thm:vaaler} for the cube, it is natural to conjecture that maximal-volume projections of the cross-polytope $B_1^n$ occur at coordinate subspaces.

\begin{conjecture}\label{conj:max-proj-cross-poly}
Fix $1 \leq k \leq n$. For every $k$-dimensional subspace $H$ in $\R^n$, we have
\[
\vol_k(\proj_H(B_1^n)) \leq \vol_k(B_1^k).
\]
\end{conjecture} 

This conjecture has appeared in this generality in Ivanov's work \cite{Iva}, who has confirmed it for $k = 2, 3$ and arbitrary $n$, using perturbation methods for frames (see also \cite{Iva-frames}). Earlier, Filliman in \cite{Fil-largest}, established the same for $k=2$, using different, more algebraic methods of his work \cite{Fil2}, also reducing the case of $k=3$ with arbitrary $n$ to $n \leq 42$. For minimal-volume projections, we have the following \emph{dual} analogue of Conjecture \ref{conj:min-sec-cube} for the cube.

\begin{conjecture}[Ivanov \cite{Iva}]\label{conj:min-proj-cross-poly}
Fix $1 \leq k \leq n$. Let $H^*$ be the $k$-dimensional subspace from Conjecture \ref{conj:min-sec-cube}. For every $k$-dimensional subspace $H$ in $\R^n$, we have
\[
\vol_k(\proj_H(B_1^n)) \geq \vol_k(\proj_{H^*}(B_1^n)).
\]
\end{conjecture}

This has been confirmed for $k=2$ by Ivanov in \cite{Iva}. As for the cube, the conjectured maximiser is an affine cross-polytope.

\subsection{Balls of $p$-norms}
The sharp results on hyperplane projections of the cross-polytope have been extended by Barthe and Naor in \cite{BN} to $\ell_p$-balls (with $p \geq 2$), thereby bringing the knowledge on extremal volume \emph{hyperplane} projections to the same level as for sections (see Tables \ref{tab:sec} and \ref{tab:proj}).

\begin{theorem}[Barthe-Naor \cite{BN}]\label{thm:BN}
For every unit vector $a$ in $\R^n$, the function
\[
p \mapsto \frac{\vol_{n-1}(\proj_{a^\perp}(B_p^n))}{\vol_{n-1}(B_p^{n-1})}
\]
is nondecreasing on $[1,+\infty)$. 
\end{theorem}

This can be viewed as a counterpart of Meyer and Pajor's Theorem \ref{thm:MP} for \emph{hyperplane} projections. It is an interesting open question to find such a monotonicity result for \emph{all} subspaces. As for the cross-polytope, it is Cauchy's formula that allows to obtain a probabilistic expression for the volume in the hyperplane case. Barthe and Naor's argument goes as follows. 

First, the surface area measure is related to the cone volume measure, by a general relation of Naor and Romik from \cite{NR}. To sketch this, let $\sigma_K$ be the normalized surface area measure on $\partial K$ and let $S$ be the not normalized surface are measure, that is $S(A)=\vol_{n-1}(\partial K \cap A) / \vol_{n-1}(\partial K)$. Let $\mu_K$ be the normalized cone volume measure, that is, for $A \subseteq \partial K$ let  $\mu_K(A) = \vol_n(\conv(\{0\} \cup A)) / \vol_n(K)$.  Let $C$ denote its not normalized version.

\begin{lemma}\label{lem:cone-vs-surface}
If $K$ is a symmetric convex body in $\R^n$, then $\sigma_K$ is absolutely continuous with respect to $\mu_K$ and for almost all $x \in \partial K$ one has
\[
	\po{\sigma_K}{\mu_K}(x) = \frac{n\vol_n(K)}{\vol_{n-1}(\partial K)} | \nabla(\|\cdot \|_K)(x)  |.
\] 
\end{lemma} 

\begin{proof}[Sketch of the proof]
For points $x$ such that $x$ is perpendicular to the surface of $K$ one has $|x| \cdot \dd S(x) = n  \dd C(x)$. If the angel between the surface and  $x$ is $\alpha$, then $|\cos \alpha| \cdot |x| \cdot \dd S(x) = n  \dd C(x)$. We clearly have $|\cos \alpha|=|\scal{n(x)}{x/|x|}|$. Let $z=\nabla \|\cdot\|_K(x)$. If $x \in \partial K$ then $1+\e = \|x+ \e x\|_K \approx \|x\|_K + \e \scal{z}{x}=1+ \e \scal{z}{x}$, which gives $\scal{z}{x}=1$.  Also, $z$ is a vector perpendicular to $\partial K$. Thus $n(x)=z / |z|$. We obtain
$
	|\cos \alpha| = \frac{1}{|x|} \cdot |\scal{n(x)}{x}| = \frac{|\scal{z}{x}|}{|x| \cdot |z|} .
$ 
This gives
\[
\frac{\vol_{n-1}(\partial K) \dd \sigma_K(x)}{|\nabla (\|\cdot\|_K)(x)|}  = 	\frac{\dd S(x)}{|\nabla (\|\cdot\|_K)(x)|} = 	\frac{|\scal{z}{x}|}{|z|}  \dd S(x) = n \dd C(x) = n \vol_n(K) \dd \mu_K(x).
\]
\end{proof}
From Lemma \ref{lem:cone-vs-surface} we therefore get
\begin{equation}\label{eq:proj-with-grad}
	|\proj_{a^\perp} K| = \frac{n}{2} \vol_n(K) \int_{\partial K} |\scal{(\nabla \|\cdot \|_K)(x)}{a}| \dd \mu_K(x),
\end{equation}
since $(\nabla \|\cdot \|_K)(x) = n(x)|(\nabla \|\cdot \|_K)(x)|$.

Second, the cone volume measure $\mu_{B_p^n}$ enjoys a probabilistic representation in terms of i.i.d. random variables, discovered by Rachev and R\"uschendorf in \cite{RR} and independently by Schechtman and Zinn in \cite{SZ}. We shall also later need a modification of the representation of the uniform measure on $B_p^n$ obtained in \cite{BGMN} by Barthe, Gu\'edon, Mendelson and Naor. Let us formulate a generalization of these results discussed in \cite{PTT18}.

\begin{lemma}\label{lem:uniform-on-bpn}
Let $K$ be a symmetric convex body and let $Z$ be any random vector in $\R^n$ with density of the form $f(\|x\|_K)$ for some continuous $f:[0,\infty) \to [0,\infty)$. Let $U$ be a random variable uniform in $[0,1]$, independent of $Z$. Then 
\begin{itemize}\itemsep=0.2cm
\item[(a)] $\frac{Z}{\|Z\|_K}$ has distribution $\mu_K$ and $U^{1/n}\frac{Z}{\|Z\|_K}$ is uniformly distributed on $K$,
\item[(b)] $\frac{Z}{\|Z\|_K}$ and $\|Z\|_K$ are independent.
\end{itemize}
In particular, for $K=B_p^n$ one can take $Z=(Y_1, \ldots, Y_n)$ where $Y_i$ are i.i.d. random variables having densities $(2\Gamma(1+\frac1p))^{-1} e^{-|t|^p}$.
\end{lemma}

\begin{proof}[Proof]
We first claim that for any integrable $h:\R^n \to \R$ the following identity holds
\begin{equation}\label{eq:polar-int}
	\int h = n |K| \int_0^\infty r^{n-1} \int_{\partial K} h(rz) \dd \mu_K(z) \dd r.
\end{equation}
To show it one can assume that $h= \1_A$, where $A=[a,b] \cdot A_0$, where $A_0 \subset \partial K$, as these sets generate the sigma algebra of Borel sets in $\R^n$. For $z \in \partial K$ and $r>0$ we then have $h(rz)=\1_{[a,b]}(r) \1_{A_0}(z)$. Thus \eqref{eq:polar-int} reduces to
\begin{equation}\label{eq:U1n}
	|A| = |K| \left(\int_a^b n r^{n-1} \dd r \right) \mu_K(A_0) = |K| (b^n-a^n) \mu_K(A_0)  = |[a,b] A_0|
\end{equation}
and is therefore true. Now, let us notice that for $\phi:\R^n \to \R$ and $\psi:\R \to \R$ we have
\[
	\E\left[ \phi\left( \frac{Z}{\|Z\|_K} \right) \psi(\|Z\|_K) \right] = \int_{\R^n}  \phi\left( \frac{x}{\|x\|_K} \right) \psi(\|x\|_K) f(\|x\|_K) \dd x = n|K| \int_0^\infty \psi(r) f(r)r^{n-1} \dd r \int_{\partial K} \phi(z) \dd \mu_K(z) 
\]
Taking $\phi, \psi \equiv 1$ we learn that $n|K| \int_0^\infty f(r) r^{n-1} \dd r =1$. Thus taking $\psi \equiv$ and next $\phi \equiv 1$ we arrive at
\[
	\E\left[ \phi\left( \frac{Z}{\|Z\|_K} \right)  \right] = \int_{\partial K} \phi(z) \dd \mu_K(z), \qquad \E\left[  \psi(\|Z\|_K) \right] = n|K| \int_0^\infty \psi(r) f(r)r^{n-1} \dd r.
\] 
The first equation shows that $\frac{Z}{\|Z\|_K}$ has distribution $\mu_K$. Moreover, we get
\[
	\E\left[ \phi\left( \frac{Z}{\|Z\|_K} \right) \psi(\|Z\|_K) \right]  = \E\left[ \phi\left( \frac{Z}{\|Z\|_K} \right)  \right] \E\left[  \psi(\|Z\|_K) \right],
\]
which shows (b). Finally \eqref{eq:U1n} to gether with the fact that $U^{1/n}$ has density $n r^{n-1}$ on $[0,1]$ shows that 
\[
	\frac{|A|}{|K|} = \p{U^{1/n} \in [a,b]} \p{\frac{Z}{\|Z\|_K} \in A_0} = \p{U^{1/n}\frac{Z}{\|Z\|_K} \in A},
\] 
which shows the second part of point (a).

\end{proof}

We can now prove the probabilistic formula for the volume of hyperplane projection of $B_p^n$.

\begin{lemma}
For $p > 1$ and every unit vector $a \in \R^n$, we then have
\begin{equation}\label{eq:Bpn-proj}
\vol_{n-1}(\proj_{a^\perp}(B_p^n)) = \frac{\vol_{n-1}(B_p^{n-1})}{\E|X_1|}\E\left|\sum_{j=1}^n a_jX_j\right|,
\end{equation}
where here $X_1, \dots, X_n$ are i.i.d. random variables with density 
$
	f_p(x) = \frac{p}{2(p-1)\Gamma(1/p)} |x|^{\frac{2-p}{p-1}} e^{-|x|^{\frac{p}{p-1}}}.
$ 
\end{lemma}

\begin{proof}
By \eqref{eq:proj-with-grad} and Lemma \ref{lem:uniform-on-bpn} (a) for some constant $c_{p,n}$ we have
\begin{align*}
	\vol_{n-1}(\proj_{a^\perp} B_p^n) & = C(p,n) \E \left|\sum_{i=1}^n a_i \left|\frac{Y_i}{S} \right|^{p-1} \sgn\left(\frac{Y_i}{S} \right) \right| = C(p,n) \cdot \frac{\E S^{p-1}}{\E S^{p-1}} \cdot \E \left|\sum_{i=1}^n a_i \left|\frac{Y_i}{S} \right|^{p-1} \sgn\left(Y_i \right) \right| \\
	& = \frac{C(p,n)}{\E S^{p-1}} \cdot \E \left|\sum_{i=1}^n a_i \left|Y_i \right|^{p-1} \sgn\left(Y_i \right) \right|.
\end{align*} 
It now suffices to observe that $X_i = \left|Y_i \right|^{p-1} \sgn\left(Y_i \right)$ for $p>1$ have densities $f_p$. We then compute $C_{p,n}$ by taking $a=e_1$.
\end{proof}

Next, Meyer and Pajor's arguments involving peakedness are replaced by the stochastic convex (Choquet) ordering, where the independence of the $X_j$ is crucial. For $p > 2$, additional structure emerges: the $X_j$ are Gaussian mixtures. This leads to an analogue of Koldobsky's Theorem \ref{thm:Kol}, the proof of which was later simplified in \cite{ENT1} by bypassing the Fourier analytic arguments (we shall discuss the arguments in Section \ref{sec:methods}).

\begin{theorem}[Barthe-Naor \cite{BN}]\label{thm:BN-p>2}
Let $p > 2$. For every two unit vectors $a$ and $b$ in $\R^n$ such that $(b_1^2, \dots, b_n^2)$ majorises $(a_1^2, \dots, a_n^2)$, we have
\[
\vol_{n-1}(\proj_{a^\perp}(B_p^n) \geq \vol_{n-1}(\proj_{b^\perp}(B_p^n)).
\]
\end{theorem}

In the range $0 < p < 1$, Cauchy's formula cannot be applied due to the lack of convexity and no nontrivial bounds are known. When $1 < p < 2$, the maximal-volume hyperplane projection is onto a coordinate subspace, as follows from Theorem \ref{thm:BN}, whereas the minimal one is not known. Barthe and Naor in \cite{BN} have shown that the cross-polytope minimiser $(\frac{1}{\sqrt{2}},\frac{1}{\sqrt{2}}, 0, \dots, 0)^\perp$ is beaten by the diagonal one for every $p > p_0=\frac{4}{3}$ in large enough dimensions (in particular, as Oleszkiewicz has pointed out in \cite{Ole}, there is no ``formal duality'' with sections, for there is not such a phase transition at $\frac{p_0}{p_0-1} = 4$).

For higher codimensions than $1$, plainly Meyer and Pajor's Theorem \ref{thm:MP} gives a sharp lower bound: for every $p \geq 2$, $1 \leq k \leq n$ and $k$-dimensional subspace in $\R^n$, we have
\[
\vol_k(\proj_H(B_p^n)) \geq \vol_k(B_p^n \cap H) \geq \vol_k(B_p^k),
\]
attained at coordinate subspaces. For $0 < p < 2$, using his reverse form of the Brascamp-Lieb inequality from \cite{Bar-BL}, Barthe in \cite{Bar-prod} has established the following lower bound.

\begin{theorem}[Barthe \cite{Bar-prod}]\label{thm:Barthe-l_p-low}
Let $0 < p < 2$. Fix $1 \leq k \leq n$. For every $k$-dimensional subspace $H$ in $\R^n$, we have
\[
\vol_k(\proj_H(B_p^n)) \geq \left(\frac{k}{n}\right)^{k(1/p-1/2)}\vol_k(B_p^k).
\]
\end{theorem}

This is optimal when $k$ divides $n$ and $p \geq 1$ (attained at subspaces from Conjecture~\ref{conj:min-sec-cube}).

\section{Methods}\label{sec:methods}

We would like to present and emphasise one particular probabilistic point of view which gathers the major results for both sections and projections under the same umbrella. The point is that as it is very natural to set up hyperplane projection problems as sharp $L_1-L_2$ comparison inequalities (thanks to Cauchy's formula, see, e.g. \eqref{eq:Bpn-proj}), the same probabilistic picture captures sections upon changing the $L_1$ norm to $L_q$ norms with negative exponents $q$.

\subsection{Sections}

This is a straighforward extension  to higher codimensions of Kalton and Koldobsky's observation made in \cite{KalKol}, recall \eqref{eq:KalKol}.

\begin{lemma}[\cite{CNT}]\label{lm:sec-formula}
Let $K$ be a body in $\R^n$ of volume $1$, star-shaped with respect to the origin. Let $H$ be a $\kk$-codimensional subspace in $\R^n$ and let $X$ be a random vector uniform on $K$. Let $\|\cdot\|$ be a norm in $H^\perp$ with the unit ball $B$. Then
\[
\vol_{n-k}(K \cap H) = \lim_{q \to -\kk+} \frac{\kk+q}{\kk\vol_{\kk}(B)}\E\|\proj_{H^\perp}X\|^{q}.
\]
\end{lemma}

\begin{proof}
If we let $f\colon H^\perp \to [0,+\infty)$ be the density of $\proj_{H^\perp}X$, as in \eqref{eq:sec-density}, we have
\[
\vol_{n-k}(K \cap H) = f(0).
\]
The function $x \mapsto \frac{\kk-q}{\kk\vol_{\kk}(B)}\|x\|^{-q}$ as $q \to -\kk+$ behaves like the Dirac delta at $0$: if $f$ is continuous at $0$ and integrable, then
\[
 \lim_{q \to -\kk+} \frac{k+q}{k\vol_{\kk}(B)} \int_{H^\perp} \|x\|^{-q}f(x) \dd x = f(0)
\]
and the lemma follows. To justify the last identity, for simplicity we identify $H^\perp$ with $\R^\kk$ and fix $\e > 0$. The set $\{x, \ f(x) < f(0) + \e\}$ contains a neighbourhood of $0$, say $\delta B$. Then
\begin{align*}
\frac{k+q}{k\vol_{\kk}(B)} \int_{\R^{\kk}} \|x\|^{q}f(x) \dd x &\leq \big(f(0)+\e\big)\frac{k+q}{k\vol_{\kk}(B)}\int_{\delta K} \|x\|^{q} \dd x + \frac{k+q}{k\vol_{\kk}(B)}\delta^q\int_{\R^k} f \\
&=\big(f(0)+\e\big)\delta^{k+q} + \frac{k+q}{k\vol_{\kk}(B)}\delta^q\int_{\R^k} f
\end{align*}
(the last equality by the homogeneity of volume and the layer cake representation). Taking $\limsup$ as $q \downarrow -\kk$ gives an upper bound by $f(0) + \e$. A lower bound is obtained similarly (the second term above can be dropped).
\end{proof}

For hyperplane sections of the cube, the limit can be evaluated which leads to a particularly handy expression.

\begin{lemma}[K\"onig-Koldobsky \cite{KK-slabs}]\label{lm:formula-cube}
Let $\xi_1, \xi_2, \dots$ be i.i.d. random vectors uniform on the sphere $S^2$ in $\R^3$. For a unit vector $a$ in $\R^n$, we have
\[
\vol_{n-1}(Q_n \cap a^\perp) = \E\left|\sum_{j=1}^n a_j\xi_j\right|^{-1}.
\]
\end{lemma}

\begin{proof}
Lemma \ref{lm:sec-formula} yields
\[
\vol_{n-1}(Q_n \cap a^\perp) = \lim_{q \to -1+} \frac{1+q}{2}\E\left|\sum_{j=1}^n a_jX_j\right|^{q},
\]
where $X = (X_1, \dots, X_n)$ is uniform on $Q_n$, that is the components $X_j$ are independent uniform on $[-\frac{1}{2}, \frac{1}{2}]$. By Archimedes' hat-box theorem, $\scal{\xi_j}{e_1}$ has the same distribution as $2X_j$, which allows to get for every fixed $q > -1$
\[
\frac{1+q}{2}\E\left|\sum_{j=1}^n a_jX_j\right|^{q} = 2^{-1-q}\E\left|\sum_{j=1}^n a_j\xi_j\right|^{q}
\]
(see, e.g. \cite{CKT} for all details). Taking the limit finishes the proof.
\end{proof}

\begin{remark}\label{rem:sec-block}
Replacing the $\xi_j$ by i.i.d. random vectors uniform on a higher dimensional sphere say $S^{d+1}$ and the exponent $-1$ by $-d$ results with a formula for sections of balls in $\ell_\infty(\ell_2)$ by block-subspaces (see Proposition 3.2 in \cite{Brz}).
\end{remark}

To illustrate the applicability of this lemma, we sketch the proof of the lower bound of Theorem \ref{thm:cube-stab}, the Hadwiger-Hensley bound with an optimal deficit.
\begin{proof}[Proof (Sketch)]
The key is to write
\[
\left|\sum_{j=1}^n a_j\xi_j\right|^2 = \sum_{i,j} a_ia_j \scal{\xi_i}{\xi_j} = 1 + 2 \sum_{i<j}a_ia_j\scal{\xi_i}{\xi_j}.
\]
The random variable $R =2\sum_{i<j}a_ia_j\scal{\xi_i}{\xi_j}$ has mean $0$. Thus by convexity,
\[
\E(1+R)^{-1/2} \geq \E(1-R/2) = 1.
\]
To improve upon this, it suffices to use a more precise pointwise inequality, say
\[
(1+r)^{-1/2} \geq 1 - \frac{1}{2}r + \frac{1}{3}r^2 - \frac{5}{24}r^3, \qquad r > -1
\]
and estimate $\E R^2$, $\E R^3$ which are explicitly expressed in terms of the $a_j$.
\end{proof}

For $B_p^n$ balls, a direct application of Lemma \ref{lm:sec-formula} leaves us with a random vector uniform on $B_p^n$ with \emph{mildly} dependent components. This however can be circumvented thanks to the homogeneity of $L_q$ norms.

\begin{lemma}[\cite{CNT}]
Let $p > 0$ and $Y_1, Y_2, \dots$ be i.i.d. random variables with density $e^{-\beta_p^p|x|^p}$, $\beta_p = 2\Gamma(1+1/p)$. Let $H$ be a subspace in $\R^n$ of codimension $k$ such that the rows of a $k \times n$ matrix $U$ form an orthonormal basis of $H^\perp$. Let $v_1, \dots, v_n \in \R^k$ denote the columns of $U$. Then
\[
\vol_{n-k}(B_p^n \cap H) = \vol_{n-k}(B_p^{n-k})\lim_{q \to -k+} \frac{k+q}{k\vol_k(B_{\|\cdot\|})}\E\left\|\sum_{j=1}^n Y_jv_j\right\|^{q},
\]
where $\|\cdot\|$ is a norm on $\R^k$ with unit ball $B_{\|\cdot\|}$.
\end{lemma}
\begin{proof}
Let $X= (X_1, \dots, X_n)$ be a random vector uniform on $B_p^n$. Lemma \ref{lm:sec-formula} then gives the desired formula with the $X_j$ in place of $Y_j$ and without the factor $\vol_{n-k}(B_p^{n-k})$. To pass to $Y$ we shall use Lemma \ref{lem:uniform-on-bpn}, which ensures that for $Y=(Y_1,\ldots, Y_n)$ and $S= (\sum_{i=1}^n |Y_i|^p)^{1/p}$ the random vector $\frac{Y}{S}$ is independent of $S$ and moreover $U^{1/n} \frac{Y}{S}$ is uniformly distributed in $B_p^n$ if $U$ is independent of $Y_i$ and uniform on $[0,1]$. Therefore
\[
\E\left\|\sum_{j=1}^n X_jv_j\right\|^{q} = \E\left\|\sum_{j=1}^n U^{1/n} \frac{Y_j}{S} v_j\right\|^{q}  = \E[U^{q/n}] \cdot \frac{\E[S^q]}{\E[S^q]} \cdot \E\left\|\sum_{j=1}^n \frac{ Y_j}{S} v_j\right\|^{q} =  \frac{\E[U^{q/n}]}{\E[S^q]} \cdot  \E\left\|\sum_{j=1}^n Y_j v_j\right\|^{q}.
\]
\end{proof}

This has been instrumental in the proof of Theorem \ref{thm:Bpn-stab}. For a simpler application, the Meyer-Pajor monotonicity result from Theorem \ref{thm:MP} holds in fact  for $L_q$ norms. In view of the previous lemma, this readily implies their theorem.

\begin{theorem}
For $p > 0$, let $Y_1^{(p)}, Y_2^{(p)}, \dots$ be i.i.d. random variables with density $e^{-\beta_p^p|x|^p}$, $\beta_p = 2\Gamma(1+1/p)$. For every vectors $v_1, \dots, v_n$ in $\R^k$ and $-k < q < 0$, we have that the function
\[
(p_1, \dots, p_n) \mapsto \E\left\|\sum_{j=1}^n Y_j^{(p_j)}v_j\right\|^q
\]
is nondecreasing in each variable.
\end{theorem}

\begin{proof}
Following Kanter, \cite{Kant}, we say for two probability measures $\mu$ and $\nu$ on $\R^n$ that $\nu$ is more peaked than $\mu$ if $\nu(K) \geq \mu(K)$ for every symmetric convex set $K$ in $\R^n$. Crucially, this is preserved by taking products and convolutions of even log-concave measures (see Corollaries 3.2 and 3.3 in \cite{Kant}). If $0 < p < p'$, then the density of $Y_1^{(p')}$ intersects the density of $Y_1^{(p)}$ exactly once and dominates it (pointwise) near the origin. Thus  $Y_1^{(p')}$ is more peaked than $Y_1^{(p)}$ and consequently $\sum Y_j^{(p_j')}v_j$ is more peaked than $\sum Y_j^{(p_j)}v_j$, if $p_j \leq p_j'$. In particular, for every $t > 0$,
\[
\p{\left\|\sum_{j=1}^n Y_j^{(p_j)}v_j\right\| \leq t} \leq \p{\left\|\sum_{j=1}^n Y_j^{(p_j')}v_j\right\| \leq t}
\]
and the result follows by integrating in $t$.
\end{proof}

The measure with density $e^{-\beta_p^p|x|^p}$ from Lemma \ref{lm:sec-formula} enjoys a \emph{Gaussian mixture} form when $0 < p < 2$. This in turn provides good convolution properties, allowing in particular to evaluate the limit from Lemma \ref{lm:sec-formula}. We say that a random variable $X$ is a (symmetric) Gaussian mixture, if $X$ has the same distribution as $RG$ for some nonnegative random variable $R$ and a standard Gaussian random variable $G$, independent of $R$. 
Gaussian mixtures are continuous, i.e. have densities and $X$ is a Gaussian mixture if and only if its density $f$ is of the form
\[
f(x) = \int_0^\infty e^{-tx^2} \dd \nu(t)
\]
for a Borel measure $\nu$ on $[0,+\infty)$. By Bernstein's theorem, this is equivalent to $g(x) = f(\sqrt{x})$ being completely monotone, that is $(-1)^{n}g^{(n)}(x) \geq 0$ for all $n \geq 0$ and $x > 0$, which gives a practical condition. We refer to \cite{ENT1} for further details and more examples. Thus, if $X_1, \dots, X_n$ are independent Gaussian mixtures, say $X_j = R_jG_j$ and $v_1, \dots, v_n$ are vectors in $\R^k$, then, conditioned on the values of the $R_j$, $\sum X_jv_j$ is a centred Gaussian random vector in $\R^k$ with covariance matrix $\sum R_j^2v_jv_j^\top$.

\begin{lemma}[\cite{ENT1, NT}]\label{lm:sec-GM}
Let $0 < p < 2$. There are nonnegative i.i.d. random variables $R_1, R_2, \dots$ such that for every subspace $H$ in $\R^n$ of codimension $k$, we have
\[
\vol_{n-k}(B_p^n \cap H) = \vol_{n-k}(B_p^{n-k})\E\left(\det\left[\sum_{j=1}^n R_jv_jv_j^\top\right]\right)^{-1/2},
\]
where $v_1, \dots, v_n$ are vectors in $\R^k$ such that the rows of the $k \times n$ matrix with columns $v_1, \dots, v_n$ form an orthonormal basis of $H^\perp$.
\end{lemma}

\begin{remark}
To describe the distribution of the $R_j$, for $0 < \alpha < 1$, we let $g_\alpha$ be the density of a standard positive $\alpha$-stable random variable $W_\alpha$, i.e. with the Laplace transform $\E e^{-uW_\alpha} = e^{-u^\alpha}$, $t > 0$ and let $V_1, V_2, \dots$ be i.i.d. random variables with density $\frac{\sqrt{\pi}}{2\Gamma(1+1/p)}t^{-3/2}g_{p/2}(t^{-1})$. Then $R_j = (\E V_j^{-1/2})^2V_j$, see \cite{ENT1}
\end{remark}

\begin{proof}[Proof of Lemma \ref{lm:sec-GM}]
The $Y_j$ from Lemma \ref{lm:sec-formula} are Gaussian mixtures, say $Y_j = T_jG_j$ for some nonnegative random variables $T_j$ and standard Gaussians $G_j$, all independent. Then, conditioned on the $T_j$, the limit in Lemma \ref{lm:sec-formula} gives the density at $0$ of the random variable $\sum Y_jv_j$ which, as we said is centred Gaussian in $\R^k$ with covariance $\sum T_j^2v_jv_j^\top$, thus its density at $0$ equals
 $(2\pi)^{-k/2}\left(\det\left[\sum_{j=1}^n T_j^2v_jv_j^\top\right]\right)^{-1/2}$.
\end{proof}

For hyperplane sections, this formula directly explains Koldobsky's Schur-convexity result from Theorem \ref{thm:Kol}.

\begin{proof}[Proof of Theorem \ref{thm:Kol}]
We first observe that if $F:\R^n \to \R$ is convex and permutation symmetric, then $F$ is Schur convex, namely $x \prec y$ implies $F(x) \leq F(y)$. Indeed, it is a standard fact (see \cite{Bh}) that there exist $(\lambda_\sigma)_{\sigma \in S_n}$ where $S_n$ stands for the set of permutations of $\{1,\ldots, n\}$ such that $\lambda_\sigma \geq 0$, $\sum_{\sigma \in S_n} \lambda_\sigma =1$ and $x = \sum_{\sigma \in S_n} \lambda_\sigma y_\sigma$, where $y_\sigma = (y_{\sigma(1)},\ldots, y_{\sigma(n)})$. Thus
\[
	F(x) = F\left( \sum_{\sigma \in S_n} \lambda_\sigma y_\sigma \right) \leq \sum_{\sigma \in S_n} \lambda_\sigma F\left(  y_\sigma \right) = \sum_{\sigma \in S_n} \lambda_\sigma F\left(  y \right) = F(y). 
\]

For a unit vector $a$ in $\R^n$, Lemma \ref{lm:sec-GM} yields
\[
\vol_{n-1}(B_p^n \cap a^\perp) = \vol_{n-1}(B_p^{n-1})\E\left(\sum_{j=1}^n a_j^2R_j\right)^{-1/2}.
\]
Since $(\cdot)^{-1/2}$ is convex, the right hand side is clearly convex and permutation symmetric ($R_j$ are i.i.d.) as a function of $(a_1^2, \ldots, a_n^2)$ and thus it is also Schur convex.

\end{proof}

\subsection{Projections}

Somewhat analogous to the Fourier-analytic approach to sections, there is a formula for the volume of hyperplane projections of a convex body as the Fourier transform of its curvature function, as discovered by Koldobsky, Ryabogin and Zvavitch in \cite{KRZ} (see also their survey \cite{KRZ-surv}). We do \emph{not} touch upon this connection here at all.
Instead, we focus on a probabilistic perspective and highlight two approaches to the $L_1-L_2$ moment comparison inequalities like \eqref{eq:Khinchin}, arising in hyperplane projections. 

As we have just seen for sections, for Gaussian mixtures, thanks to their good additive structure, we readily get precise Schur-majorisation type results. This proof is from \cite{ENT1}.

\begin{proof}[Proof of Theorem \ref{thm:BN-p>2}]
Recall  formula \eqref{eq:Bpn-proj} for hyperplane projections. For $p > 2$, the density  $f_p(t)$ of $X_i$  is completely monotone, thus the $X_j$ are Gaussian mixtures, say $X_j = R_jG_j$ for some i.i.d. nonnegative random variables $R_j$ and standard Gaussians $G_j$, all independent. Then, adding the Gaussians first conditioning on the $R_j$ yields
\begin{equation}\label{eq:L1-GM}
\E\left|\sum_{j=1}^n a_jX_j\right| = \E\left(\sum_{j=1}^n a_j^2R_j^2\right)^{1/2}\E|G_1|.
\end{equation}
As in the proof for sections, the Schur-concavity result follows from the concavity of $(\cdot)^{1/2}$.
\end{proof}

The same argument bluntly extends to arbitrary $L_q$ norms, giving sharp Khinchin inequalities (see \cite{AH} and \cite{ENT1}). 

When $1 \leq p < 2$, the density of the $X_j$ in \eqref{eq:Bpn-proj} is bimodal and understanding the $L_1$ norm of their weighted sums is elusive, mainly due to complicated cancellations -- the problem which completely disappears in \eqref{eq:L1-GM}. For $p=1$ the $X_j$ become discrete (symmetric random signs). We present two completely different Fourier-analytic proofs. The first proof, due to Haagerup, is in the same spirit as Ball's proof from \cite{Ball} for hyperplane cube-sections.

\begin{proof}[Proof of Theorem \ref{thm:Szarek} (Haagerup \cite{Haa})]
We want to minimise $\E|\sum a_j\e_j|$ subject to $\sum a_j^2=1$. We can assume that all $a_j$ are positive. If at least one exceeds $\frac{1}{\sqrt{2}}$, say $a_1 > \frac{1}{\sqrt{2}}$, we get by averaging over the other coefficients that
\[
\E\left|\sum a_j\e_j\right| \geq \E_{\e_1}\left|a_1\e_1 + \E\sum_{j>1}a_j\e_j\right| = a_1 > \frac{1}{\sqrt{2}},
\] 
as desired. Now we assume that for all $j$, $a_j \leq \frac{1}{\sqrt{2}}$.
A starting point is the Fourier-analytic formula,
\[
|x| = \frac{1}{\pi}\int_{\R} (1-\cos(tx))t^{-2}\dd t, \qquad x \in \R.
\]
Thus, for an integrable random variable $X$,
\[
\E|X| = \frac{1}{\pi}\int_{\R} \left(1-\text{Re}(\E e^{itX})\right)t^{-2}\dd t
\]
(see also Lemmas 2.3 and 4.2 in \cite{Haa} as well as Lemma 3 in \cite{GF}).
In particular, using independence and $\E e^{it\e_j} = \cos t$, we get
\[
\E\left|\sum a_j\e_j\right| = \frac{1}{\pi}\int_{\R} \left(1 - \prod \cos(ta_j)\right)t^{-2}\dd t
\]
By the AM-GM inequality, this gives the following bound
\[
\E\left|\sum a_j\e_j\right| \geq \sum a_j^2F(a_j^{-2})
\]
with
\[
F(s) = \frac{1}{\pi}\int_{\R} \left(1 - \left|\cos\left(\frac{t}{\sqrt{s}}\right)\right|^s\right)t^{-2} \dd t, \qquad s > 0,
\]
cf. \eqref{eq:Ball-Holder} and the ensuing function $\Psi$ in Ball's proof.
Here however, function $F$ can be expressed explicitly. Using $\sum_{n=-\infty}^\infty \frac{1}{(t+n\pi)^2} = \frac{1}{\sin^2 t}$, we arrive at
\begin{align*}
F(s) = \frac{1}{\pi \sqrt{s}}\int_{\R} \left(1 - \left|\cos t\right|^s\right)t^{-2} \dd t &= \frac{1}{\pi \sqrt{s}}\sum_{n=-\infty}^\infty \int_{-\pi/2}^{\pi/2}\left(1 - \left(\cos t\right)^s\right)(t+n\pi)^{-2}\dd t \\
&=\frac{1}{\pi \sqrt{s}}\int_{-\pi/2}^{\pi/2}\left(1 - \left(\cos t\right)^s\right)\sin^{-2} t\dd t \\
&= \frac{2}{\sqrt{\pi s}}\frac{\Gamma\left(\frac{s+1}{2}\right)}{\Gamma\left(\frac{s}{2}\right)}.
\end{align*}
\textbf{Claim.} $F(s)$ increases on $(0,+\infty)$.

Using this claim and that $a_j \leq \frac{1}{\sqrt{2}}$ for all $j$, we finish the proof,
\[
\E\left|\sum a_j\e_j\right| \geq \sum a_j^2F(a_j^{-2}) \geq \sum a_j^2F(2) = F(2) = \frac{1}{\sqrt{2}}.
\]
Noteworthy, this is tight when $n=2$ and $a_1 = a_2 = \frac{1}{\sqrt{2}}$.

To show the claim, we note that $\lim_{s \to \infty} F(s) = \sqrt{\frac{2}{\pi}}$ (e.g. by Stirling's formula) and that
\[
F(s+2) = \sqrt{\frac{s}{s+2}}\frac{s+1}{s}F(s) = \left(1-1/(s+1)^2\right)^{-1/2}F(s)
\]
which iterated yields $F(s+2n) = F(s)\prod_{k=0}^{n-1}\left(1-1/(s+2k+1)^2\right)^{-1/2}$, so letting $n \to \infty$,
\[
F(s) = \sqrt{\frac{2}{\pi}} \prod_{k=0}^{\infty}\left(1-1/(s+2k+1)^2\right)^{1/2}.
\]
\end{proof}

The second proof  uses the machinery of Fourier analysis on the discrete cube $\{-1,1\}^n$. We refer for instance to Chapter 1 in \cite{OD} for basic background.

\begin{proof}[Proof of Theorem \ref{thm:Szarek} (Kwapie\'n-Lata\l a-Oleszkiewicz \cite{KLO, LO-best, O1})]
We work with $L_2(\{-1,1\}^n, \R)$ equipped with the product probability measure on the cube $\{-1,1\}^n$, i.e. the distribution of $(\e_1, \dots, \e_n)$ and the inner product $\scal{f}{g} = \E \big[f(\e)g(\e)\big]$, $f, g\colon \{-1,1\}^n\to\R$. Let
\[
f(x) = \left|\sum_{j=1}^n a_jx_j\right|, \qquad x \in \{-1,1\}^n.
\]
We write its discrete Fourier expansion with respect to the orthonormal system of the Walsh functions $w_S(x) = \prod_{j \in S} x_j$ indexed by the subsets $S \subset \{1,\dots,n\}$ with $w_\varnothing(x) \equiv 1$. We have,
\[
f(x) = \sum_{S} b_Sw_S(x), \qquad b_S = \scal{f}{w_S}.
\]
Since $f$ is even, $b_S = 0$ provided $|S|$ is odd. The crux is to consider the Laplace operator $\mathcal{L}$ acting on $L_2(\{-1,1\}^n, \R)$,
\[
(\mathcal{L}g)(x) = \frac{1}{2}\sum_{y \sim x} (g(y)-g(x))
\]
where the sum is over all neighbours $y$ of $x$, i.e. the points in $\{-1,1\}^n$ differing from $x$ by one component. As can be checked, the Walsh functions are its eigenfunctions, $\mathcal{L} w_S = -|S|w_S$ and for \emph{even} functions $g$, we have the following Poincar\'e-type inequality,
\[
\scal{g}{-\mathcal{L}g} \geq 2\var(g).
\]
\textbf{Claim.} $(-\mathcal{L}f)(x) \leq f(x)$ for every $x \in \{-1,1\}^n$.

Using this claim in the Poincar\'e inequality,
\[
2\big(\E f^2 - (\E f)^2\big) \leq \scal{f}{-\mathcal{L}f} \leq \scal{f}{f} = \E f^2
\]
which gives $\E f \geq \frac{1}{\sqrt{2}} (\E f)^2$, as desired. The claim follows from rearranging the following consequence of the triangle inequality,
\begin{align*}
&|-a_1x_1 + a_2x_2 + \dots + a_nx_n| + |a_1x_1 - a_2x_2 + \dots + a_nx_n| \\
&\qquad+ \dots + |a_1x_1 + a_2x_2 + \dots - a_nx_n| \geq (n-2)|a_1x_1 + \dots + a_nx_n|.
\end{align*}
\end{proof}

We stress out that this proof is extremely robust: it only uses the triangle inequality and hence extends verbatim to the case where the coefficients $a_j$ are vectors in an arbitrary normed vector space.

The history of this argument is a bit convoluted. Lata\l a and Oleszkiewicz's work \cite{LO-best} contains all the crucial ideas of the modern proof presented above, however, it is not written in a Fourier analytic language. The proof presented here was devised by Kwapie\'n and is based on the Walsh functions (the characters of $\{-1,1\}^n$). As we have seen, one of its main components is a strengthened Poincar\'e-type inequality in the presence of symmetry, the idea of which appeared first in \cite{KLO} (in the continuous case), extended to the discrete case in \cite{O1} (perhaps the first place where this proof appears in print). Oleszkiewicz presented this proof in 1996 at MSRI (during a workshop in harmonic analysis and convex geometry).

We finish with a sketch of the Barthe-Naor proof from \cite{BN} of the monotonicity result from Theorem \ref{thm:BN} featuring yet another tool, useful in proving Khinchin-type inequalities: the stochastic convex ordering. This circle of ideas was further developed in \cite{ENT2}. 

In the simplest setting sufficient for our purposes, for two symmetric random variables $X$ and $Y$, we say that $Y$ \emph{dominates} $X$ in the convex (or often called Choquet) stochastic ordering, if $\E\phi(X) \leq \E\phi(Y)$ for every even convex function $\phi\colon\R \to [0,+\infty]$. It is clear that this tensorises and is preserved by convolution: if $Y$ dominates $X$ and $Z$ is a symmetric random variable, independent of them, then $Y+Z$ dominates $X+Z$. We will only need the following sufficient condition.

\begin{lemma}\label{lm:choq-suff}
If random variables $X$ and $Y$ satisfy $\E |X| = \E |Y|$, have even densities $f$ and $g$ respectively and there are $0 < x_1 < x_2$ such that $\{t \geq 0, \ g(t) < f(t)\}$ is the interval $(x_1, x_2)$ ($f$ and $g$ \emph{intersect} twice), then $Y$ dominates $X$ in the convex stochastic ordering.
\end{lemma} 
\begin{proof}
Let $\phi\colon \R \to [0,+\infty]$ be an even convex function. Thanks to the symmetry of $X, Y$ and the constraint $\E|X| = \E|Y|$, the desired inequality $\int \phi f \leq \int \phi g$ is equivalent to 
\[
\int_0^\infty \big(\phi(x)-\alpha x - \beta\big)\big(g(x) - f(x)\big) \dd x \geq 0
\]
with some (any) $\alpha, \beta \in \R$. We choose $\alpha, \beta$ as the unique parameters such that the convex function $\psi(x) = \phi(x)-\alpha x - \beta$ vanishes at $x_1$ and $x_2$. Then, by convexity, $\psi\leq 0$ on $(x_1, x_2)$ and $\psi \geq 0$ outside that interval. Thus the integrand is pointwise nonnegative.
\end{proof}

\begin{proof}[Proof of Theorem \ref{thm:BN}]
In view of \eqref{eq:Bpn-proj}, we aim at showing that the function
\[
p \mapsto \frac{1}{\E|X_1^{(p)}|}\E\left|\sum a_jX_j^{(p)}\right|
\]
is nondecreasing on $[1,+\infty)$, where the $X_j^{(p)}$ are i.i.d. random variables with density proportional to $|x|^{\frac{2-p}{p-1}}\exp\left\{-|x|^{\frac{p}{p-1}}\right\}$. By the tensorisation property, it suffices to prove that for $1 \leq p < q$, $X_1^{(q)}/\E|X_1^{(q)}|$ dominates $X_1^{(p)}/\E|X_1^{(p)}|$. This readily follows from Lemma \ref{lm:choq-suff}.
\end{proof}

\section{Other connections}

We close this survey with two tangential topics related to sections: an application of Ball's cube slicing inequality to entropy power inequalities and a reformulation of the conjectural logarithmic Brunn-Minkowski inequality in terms of sections of the cube.

\subsection{Entropy power inequalities}

Recall \eqref{eq:sec-density}, viz. the volume of central hyperplane section by $a^\perp$ is the maximum value of the density of the marginal $\scal{a}{X} = \sum a_jX_j$ (there $f(0) = \|f\|_\infty$ by the symmetry and log-concavity of $X$). The maximum density functional 
\[
M(X) = \|f\|_\infty
\]
of a random vector $X$ in $\R^n$ with density $f$ is closely related to classical topics in probability such as the L\'evy concentaration function, small ball estimates and anticoncentration, as well as information theory, particularly the entropy power inequalities. We refer to the comprehensive surveys \cite{MMX, NV}. 
The entropy power inequality originated in Shannon's seminal work \cite{Sha} and asserts that the entropy power
\[
N(X) = \exp\left\{\frac{2}{n}h(X)\right\}, \qquad h(X) = -\int_{\R^n} f \log f
\]
is superadditive: for \emph{independent} random vectors $X$ and $Y$ in $\R^n$, we have
\[
N(X+Y) \geq N(X) + N(Y),
\]
and plainly, by induction, the same for arbitrarily many independent summands. In analogy, we let 
\[
N_\infty(X) = \exp\left\{\frac{2}{n}h_\infty(X)\right\} = M(X)^{-2/n}, \qquad h_\infty(X) = -\log\|f\|_\infty
\]
be the $\infty$-entropy power of $X$, sometimes called the min-entropy power (because for a fixed distribution, it is the smallest entropy power across the family of all R\'enyi entropies). The min-entropy power inequality in dimension $1$ reads as follows.

\begin{theorem}[Bobkov-Chistyakov \cite{BCh}, Melbourne-Roberto \cite{MR}]\label{thm:BobChi}
Let $X_1, \dots, X_m$ be independent random variables with bounded densities. Then,
\[
N_\infty\left(X_1+\dots+X_m\right) \geq \frac{1}{2}\sum_{j=1}^m N_\infty(X_j)
\]
with equality if and only if two of these variables are uniform on $A$, $c-A$ respectively for some set $A$ in $\R$ of finite measure and some $c \in \R$, while the other variables are constant.
\end{theorem}

Bobkov and Chistyakov proved this inequality with the sharp constant $\frac{1}{2}$ using Ball's cube-slicing inequality, whereas the equality conditions have recently been established by Melbourne and Roberto using their robust version (see Theorem \ref{thm:cube-stab}).

The argument rests on the following subtle comparison due to Rogozin.

\begin{theorem}[Rogozin \cite{Rog}]
Let $X_1, \dots, X_m$ be independent random variables with bounded densities and let $U_1, \dots, U_m$ be independent uniform random variables on intervals chosen such that $M(X_j) = M(U_j)$ for each $j$. Then
\[
M(X_1+\dots+X_m) \leq M(U_1+\dots+U_m).
\]
\end{theorem}

Theorem \ref{thm:BobChi} then follows by invoking Ball's Theorem \ref{thm:ball-codim1} which, after incorporating the variance constraint amounts to 
\[
M(U_1+\dots+U_m) \leq \sqrt{2}\left(M(U_1)^{-2}+\dots+M(U_m)^{-2}\right)^{-1/2}.
\]
In \cite{MMX-Rog}, Madiman, Melbourne and Xu have developed multivariate generalisations of Rogozin's result where the extremal distributions are uniform on Euclidean ball. They have combined it with Brzezi\'nski's bound \eqref{eq:Brzez} to obtain an extension of Theorem \ref{thm:BobChi} to $\R^n$-valued random vectors with the sharp constant $\frac{1}{2}$ replaced by $\frac{\Gamma(1+n/2)^{2/n}}{(1+n/2)}$, asymptotically sharp (as $m \to \infty$). Previously, using a different argument exploiting Young's inequalities with sharp constants, Bobkov and Chistyakov in \cite{BCh-epi} obtained such an extension with a slightly worse constant $\frac{1}{e}$ (``attained'' as $n \to \infty$), whereas in \cite{RS}, Ram and Sason have obtained constants dependent on the number of summands. Another direction, related to higer dimensional marginals, have been explored by Livshyts, Paouris and Pivovarov in \cite{LPP}.

We end this subsection with a conjectural entropic Busemann-type result.

\begin{conjecture}[Ball-Nayar-Tkocz \cite{BNT}]
Let $X$ be a symmetric log-concave random vector in $\R^n$. Then
\[
v \mapsto \sqrt{N(\scal{v}{X})} = e^{h(\scal{v}{X})}
\]
defines a norm on $\R^n$.
\end{conjecture}

Note that Busemann's Theorem \ref{thm:Bus} is equivalent to this statement with $N_\infty(\cdot)$ in place of the entropy power $N(\cdot)$ (for uniform distributions on symmetric convex bodies which generalises to all symmetric log-concave distributions by Ball's results from \cite{Ball-lc}). What supports this conjecture is the fact that $\sqrt{N(\scal{v}{X})}$ defines an $e$-quasinorm which is also a $\frac{1}{5}$-seminorm (see \cite{BNT}), as well as the conjecture holds for the R\'enyi entropy power of order $2$ (see \cite{Li}). For extensions to $\kappa$-concave measures, see \cite{MMX}.

\subsection{The logarithmic Brunn-Minkowski conjecture}
In \cite{BLYZ-ineq}, B\"or\"oczky, Lutwak, Yang and Zhang  have conjectured a strengthening of the Brunn-Minkowski inequality in the presence of symmetry and convexity, namely
\[
\vol_n(M_\lambda(K,L)) \geq \vol_n(K)^\lambda\vol_n(L)^{1-\lambda}.
\]
for every symmetric convex sets $K$ and $L$ in $\R^n$ and every $0 \leq \lambda \leq 1$, where $M_\lambda(K,L)$ is the intersection of the symmetric strips 
\[
S_\theta = \{x \in \R^n, \ |\scal{x}{\theta}| \leq h_K(\theta)^\lambda h_L(\theta)^{1-\lambda}\}
\]
over all unit vectors $\theta$ in $S^{n-1}$. Here, as usual $h_K(\theta) \sup_{y \in K} \scal{\theta}{y}$ denotes the support functional of $K$. Still resisting significant efforts of many reseachers over a decade, this far reaching conjecture stems from the so-called logarithmic Minkowski problem (see \cite{BLYZ-prob}) and we refer to E. Milman's recent work \cite{Mil-logBM} for further comprehensive background, references and best results up to date. Relevant to us is an equivalent formulation in terms of a certain convexity property of volumes of sections of rescaled cubes.

\begin{conjecture}\label{conj:logBM}
Let $1 \leq k \leq n$. For every $k$-dimensional subspace $H$ in $\R^n$, the function
\[
(t_1, \dots, t_n) \mapsto \vol_{k}\big(\text{diag}(e^{t_1},\dots,e^{t_1})B_\infty^n \cap H\big)
\]
is log-concave on $\R^n$.
\end{conjecture}

For precise statements and explanations of equivalences for this and similar formulations, we refer to \cite{NT, Sar1, Sar2}. Here, as usual $\text{diag}(e^{t_1},\dots,e^{t_n})$ is the $n\times n$ diagonal matrix with the diagonal entries $e^{t_1}, \dots, e^{t_n}$, so that $\text{diag}(e^{t_1},\dots,e^{t_1})B_\infty^n = [-e^{t_1},e^{t_1}]\times\dots\times[-e^{t_n},e^{t_n}]$. In fact, we conjecture that the conjecture remains true with $B_p^n$ in place of the cube $B_\infty^n$ for every $p \geq 1$ and have been able to verify this for $p=1$ in \cite{NT} using Lemma \ref{lm:sec-GM}.

\end{document}